\documentclass[a4paper,11pt]{article}

\usepackage{pstricks,pstricks-add,pst-math,pst-xkey}
\usepackage{amssymb,amsmath, amsthm}
\usepackage[T1]{fontenc}
\usepackage[latin1]{inputenc}
\usepackage[english]{babel}
\usepackage{amsfonts,amsbsy}
\usepackage{enumerate}
\usepackage{mathrsfs}
\usepackage{graphicx,psfrag}
\usepackage{epic}
\usepackage{stmaryrd}
\usepackage{longtable}
\usepackage[backref=page,bookmarksopen=true,dvipdfm]
{hyperref}

\theoremstyle{plain}
\newtheorem{proposition}{Proposition}[section]
\newtheorem{theorem}[proposition]{Theorem}
\newtheorem{lemma}[proposition]{Lemma}

\theoremstyle{definition}
\newtheorem{definition}[proposition]{Definition}

\theoremstyle{remark}

\newtheorem{remark}[proposition]{Remark}

\theoremstyle{plain}

\theoremstyle{definition}

\theoremstyle{remark}

\renewenvironment{proof}{\smallskip\noindent\emph{\textbf{Proof.}}\hspace{1pt}}%
{\hspace{-5pt}{\nobreak\quad\nobreak\hfill\nobreak$\square$\vspace{8pt}%
\par}\smallskip\goodbreak}

\newenvironment{proofof}[1]{\smallskip\noindent\emph{\textbf{Proof of #1.}}%
\hspace{1pt}}{\hspace{-5pt}{\nobreak\quad\nobreak\hfill\nobreak%
$\square$\vspace{8pt}\par}\smallskip\goodbreak}

{\hspace{-5pt}{\nobreak\quad\nobreak\hfill\nobreak$\square$\vspace{8pt}%
\par}\smallskip\goodbreak}

\renewcommand{\leq}{\leqslant}
\renewcommand{\geq}{\geqslant}

\newcommand{\Id}{\mathinner{\mathrm{Id}}}

\newcommand{\C}[1]{\mathscr{C}^{#1}}
\newcommand{\Cc}[1]{\mathscr{C}_c^{#1}}
\newcommand{\modulo}[1]{{\left|#1\right|}}
\newcommand{\norma}[1]{{\left\|#1\right\|}}
\newcommand{\Ref}[1]{{\rm(\ref{#1})}}
\newcommand{\reali}{{\mathbb{R}}}
\newcommand{\naturali}{{\mathbb{N}}}
\newcommand{\rpic}{{\mathbb{R}_+}}
\newcommand{\rpis}{{\mathbb{R}_+^*}}

\renewcommand{\epsilon}{\varepsilon}
\renewcommand{\phi}{\varphi}
\renewcommand{\L}[1]{{\mathbf{L}^#1}}
\renewcommand{\H}[1]{{\mathbf{H}^{{#1}}}}

\renewcommand{\div}{\mathinner{\rm div}}

\newcommand{\spt}{\mathop{\rm Supp}}

\newcommand{\pt}{\partial}

\newcommand{\lb}{\llbracket}
\newcommand{\rb}{\rrbracket}

\renewcommand{\d}[1]{\mathinner{\mathrm{d}{#1}}}
\newcommand{\D}[1]{\mathinner{\mathrm{D}{#1}}}

\newcommand{\g}{{\gamma_0}}
\newcommand{\e}{{\nu}}

\makeatletter

\@addtoreset{equation}{section}
\makeatother

\title{Global smooth solutions of Euler equations for Van der Waals gases}
\author{Magali Mercier\thanks{Universit\'e de Lyon, Universit\'e Lyon~1, \'Ecole
    centrale de Lyon, INSA de Lyon, CNRS, UMR5208, Institut Camille
    Jordan, 43 blvd du 11 novembre 1918, F-69622 Villeurbanne-Cedex,
    France}}
\begin{document}
\maketitle

\begin{abstract}
We prove  global in time existence of solutions of the Euler compressible equations for a Van der Waals gas when the density is small enough in $\H{m}$,  for $m$ large enough. To do so,  we introduce a specific symmetrisation  allowing areas of null density. Next, we make estimates in $\H{m}$, using for some terms the estimates done by M. Grassin, who proved the same theorem in the easier case of a perfect polytropic gas. We treat the remaining terms separately, due to their non-linearity.

  \medskip

\noindent\textit{2000~Mathematics Subject Classification:} 35L60, 35Q31, 76N10.

\medskip


\end{abstract}

\section{Introduction}
We are interested in the Cauchy problem for Euler compressible equations,  describing the evolution of a gas whose thermodynamical and kinetic properties are known at time $t=0$. 

More specifically, we are concerned with the life span of smooth solutions.  Various authors, in particular T. C. Sideris \cite{Sideris85, Sideris97}, T. Makino, S. Ukai \& S. Kawashima \cite{MakinoUK}, J.-Y. Chemin \cite{chemin, chemin2} have given criteria for mathematical explosion. We know also that there exist global in time solutions for well chosen initial data.
Li Ta Tsien \cite{li}, D. Serre \cite{Sr97}, M. Grassin \cite{Grass} prove for example the global in time existence of regular solutions under some hypotheses of ``expansivity''.

Most of these results were obtained within the framework of Perfect Polytropic Gases. A natural question  is to determine whether these results  extend to more realistic gases, following for example the Van der Waals law. This law takes  into account the volume of  molecules, which is important  in physical situations like explosions or implosions.  In  such limits, the gas is highly compressed and the Van der Waals law fits better with the behaviour of real gases than the Perfect polytropic gases in such limits. 
The Van der Waals law is also used to modelise dusty gases, seen as  perfect gases  with dust pollution \cite{dustgret,dustsaff,steiner,dustvish}. This law is given by the relationship
\[
p(v-b)=\mathfrak{R}T\,,
\]
where $p$ is the pressure, $v$ the massic volume, $T$ the temperature and $b, \mathfrak{R}$ are given constants.

The addition of the covolume $b$, representing the compressibility limit of the molecules in the gas, modifies non-trivially the analysis of the Euler  
equations.  However, we are going to show the global in time existence of regular solutions thus generalising a theorem of M.~Grassin \cite{Grass}.

\begin{theorem}\label{thm:VdW}
Let $m>1+d/2$. Let $(\rho_0,u_0,s_0)$  be the initial conditions for the Cauchy problem  associated to the Euler compressible equations (\ref{eq:euler})  for a Van der Waals  gas  with constant $c_v$ and $c_v>0$. Let us assume  $0\leq \rho_0\leq 1/b$. Then  we can define $\g=1+\frac{ \mathfrak{R} }{c_v}$; furthermore there exists $\varepsilon_0>0$ such that if 
\begin{description}
\item[(H1)]  $\norma{(\pi_0, s_0)}_{\H{m}(\reali^d)}\leq \varepsilon_0$, where $\pi_0= \left(\frac{\rho_0}{1-b\rho_0}\right)^{\frac{\g-1}{2}} \exp(\frac{s_0}{c_v})$,
\item[(H2)] the initial speed $u_0$ belongs to the space $ X=\{z:\rpic\times \reali^d\to \reali^d;\D{z}\in \L\infty, \D{}^2 z\in \H{m-1}\}$,
\item[(H3)] there exists $\delta>0$ such that for all $ x \in \reali^d$, $\mathrm{dist}(\mathrm{Spec}(\D{u_0})(x),\reali_-)\geq \delta$,
\item[(H4)] the initial density $\rho_0$ and the initial entropy  $s_0$  have compact support,
\end{description}
then the problem
 \begin{equation}\label{eq:app}
 \left\{\begin{array}{cccl}
 \pt_t \bar{u} +(\bar{u}\cdot \nabla)\bar{u}&=&0 & \textrm{on } \reali_+\times \reali^d,\\
         \bar{u}(0,x)&=&u_0(x)& \textrm{on } \reali^d.
     \end{array}\right.
 \end{equation}
admits a global classical solution. If furthermore $\g=\frac{\e+1}{\e-1}$ with $\e\in \naturali$ and $\e\geq 2$, or if $\g$ and $m$ satisfy $\e=\frac{\g+1}{\g-1}\geq  m>1+\frac{d}{2}$,  then there exists a global classical solution $(\rho ,u,s)$ to  the Euler compressible equations (\ref{eq:euler}) satisfying
\[
( \left(\frac{\rho}{1-b\rho}\right)^{\frac{\g-1}{2}}, u-\overline u, s)\in \C0(\rpic;\H{m}(\reali^d;\reali^{d+2}))\cap \C1(\rpic;\H{m-1}(\reali^d;\reali^{d+2}))\,.
\]
\end{theorem}

To do so, we have first to extend to Van der Waals gases a symmetrisation obtained by Makino, Ukai \& Kawashima \cite{MakinoUK}, which allows null density areas. Next we will derive energy estimates in $\H{m}$. 

Since the Van der Waals gases have a behaviour close to perfect polytropic gases for weak densities, this result is not much surprising. However, the non linear terms introduced by the  Van der Waals law have to be  treated carefully.

In Section \ref{sec:thermo}, we  describe the thermodynamic properties of a compressible gas and we state some important properties such as the Friedrichs' symmetrisation. In Section \ref{sec:proof}, we give the detailed proof of this result and in Section \ref{sec:tec} we give the proofs of some technical lemmas used in Section \ref{sec:proof}.

\section{Thermodynamic and first properties}\label{sec:thermo}
\subsection{Conservation law}
Compressible fluid dynamics, without viscosity or heat transfer, is described by the Euler equations, which are made of the conservation of mass, of momentum and of energy (see \cite[chap. 2]{Groot}) : 
\begin{equation}\label{eq:euler}
\left\{
\begin{array}{l}
\pt_t \rho + \mathrm{div}(\rho {u})=0 \,,\\
\pt_t  q + \mathrm{div}(\rho
{u}\otimes {u}) + \nabla p =0\,, \\
\pt_t E +
\mathrm{div}\left(\left(E
+ p\right){u} \right)=0\,,
\end{array}
\right. 
\end{equation}
where $\rho $  is the mass of the fluid per unit of volume,  $q=\rho u$ is the momentum per unit of volume and $E=\frac{1}{2}\rho u^2 +\rho e$  is the total energy per unit of volume, sum of the kinetic energy and internal energy. This is a system of  $(d+2)$ equations and $(d+3)$ unknowns:  the density  $\rho \in \rpic$,  the speed $u\in \reali^d$, the internal energy  $e\in \reali$ and the pressure $p\in \reali$.  In order to complete this system, we have to add a state law, for example an incomplete state law, also called \emph{pressure law} $(\rho,e )\mapsto p(\rho, e)$.

\begin{definition}
We call \emph{Euler compressible equations} the system made of the conservation laws (\ref{eq:euler}) and of an incomplete state law $p=p(\rho, e)$.
\end{definition}

A simplified model is often considered, conserving only the conservation of mass and momentum, assuming that the fluid is isentropic. This simplified system is 
\begin{equation}\label{eq:euler_isentropique}
\left\{
\begin{array}{l}
\pt_t \rho + \mathrm{div}(\rho {u})=0\,, \\
\pt_t  q + \mathrm{div}(\rho
{u}\otimes {u}) + \nabla p =0\,, \\
\end{array}
\right. 
\end{equation}
and the state law is a given function $\rho \mapsto p(\rho)$. 

\begin{definition}
We call \emph{isentropic Euler equations} the system (\ref{eq:euler_isentropique})  with a given state law $p=p(\rho)$ such that
\[
p=-\left.\frac{\pt e}{\pt v}\right|_s\;,
\]
where $v=1/\rho$  is the specific volume, $T$ is the temperature, $s$ is the specific entropy and   $f=e-Ts$  is the specific free energy, assuming  we are given a complete  equation of state $(v,s)\mapsto e(v,s)$.
\end{definition}
  The thermodynamical variables $v, s, e, T, p$  must satisfy some relations described in Section \ref{sec:lois}.
The Euler equations are a system of first order conservation law, whose study is developed in particular in the books \cite{DafermosBook,Sr1, Sr2}.

\subsection{State law}\label{sec:lois}
The state law has a strong influence on the mathematical analysis of the compressible Euler equations. The state law of ``real'' gases can reveal  particular behaviour and introduce existence and/or uniqueness troubles which do not appear for perfect gases, see \cite{MenikoffPlohr}. We describe below the physical principles  a state law has to satisfy.

\subsubsection{Definitions}
We consider a fluid, whose internal energy is a regular  function of its specific volume\footnote{specific is a synonym of  massic} $v=1/\rho$ and of its specific entropy $s$.  This means that the gas is entitled with a \emph{complete} state law, or energy law $e=e(v,s)$. The fundamental thermodynamic principle is then
\begin{equation}\label{eq:fond}
\mathrm{d}e=-p\mathrm{d}v +T\mathrm{d}s\,
\end{equation}
where $p$ is the pressure and $T$ the temperature of the gas. Consequently, the pressure $p$ and the temperature $T$ can be defined as
\begin{align}
p&=-\left.\frac{\pt e}{\pt v}\right|_s\,,&
T&=\left.\frac{\pt e}{\pt s}\right|_v\,, \label{eq:pT}
\end{align}
where the notation $|$ precises the variable maintained constant in the partial derivation.

The greater order derivatives of $e$ have also an important role; we introduce the following adimensional quantities:
\begin{align}
\gamma&=-\dfrac{v}{p}\left.\dfrac{\pt p}{\pt v}\right|_s \,, &
\Gamma&=-\frac{v}{T}\left.\frac{\pt T}{\pt v}\right|_s\,, &
\delta&=\frac{pv}{T^2}\left.\frac{\pt T}{\pt s}\right|_v\,, &
\mathscr{G}&=-\frac{v}{2}\frac{\left.\frac{\pt^3 e}{\pt v^3}\right|_s}{ \left. \frac{\pt ^2 e}{\pt v^2} \right|_s}\,.
\label{eq:scrG}
\end{align}
The  coefficient $\gamma$ is called the \emph{adiabatic exponent},  and  $\Gamma$ is the  \emph{Grüneisen coefficient}.
The quantities $\gamma,\delta, \Gamma$ and $\mathscr{G}$ characterise the geometrical properties of the isentropic curves  in the $(v,p)$ plan (see \cite{MenikoffPlohr}). They can be expressed in function of $e$ through the relationships:
\begin{align*}
\gamma&= \frac{v}{p}\frac{\pt^2 e}{\pt v^2}\,,
&\Gamma&= -\frac{v}{T}\frac{\pt^2 e}{\pt s\pt v}\,, &
 \delta&=\frac{pv}{T^2}\frac{\pt^2 e}{\pt s^2}\,.
\end{align*}
 
We also introduce the \emph{calorific capacity at constant volume} $c_v$   and the \emph{calorific capacity at constant pressure} $c_p$ by
\begin{align}
c_v&=\left.\dfrac{\partial e}{\pt T}\right|_v =\frac{T}{\left.\frac{\pt^2 e}{\pt s^2}\right|_v}\,, &
c_p&=T\left.\frac{\pt s}{\pt T}\right|_p\,. \label{eq:cp}
\end{align}
These two quantities are linked with  $\frac{p v}{T}$ and with $\gamma$, $\delta$, $\Gamma$ through
\begin{equation}\label{eq:deltacv}
 \delta c_v=\frac{pv}{T}\,,\qquad\qquad c_p=\frac{pv}{T}\frac{\gamma}{\gamma \delta -\Gamma^2}\,.
\end{equation}

The quantity $\gamma_*=\frac{c_p}{c_v}$ can besides be expressed as  $\gamma_*=\frac{\gamma \delta}{\gamma\delta -\Gamma^2}$.  It is not equal to   $\gamma$  in the general case, but for an ideal gas we have $\delta=\Gamma=\gamma-1$, so that $\gamma_*=\gamma$.

\subsubsection{Thermodynamical constraints}\label{sec:constraint}
It is very natural to assume $v\geq 0$.  We assume furthermore  that the pressure $p$ and the temperature  $T$ are positive, which imposes that $e$ is a function increasing in $T$ and decreasing in $v$.

A classical thermodynamical hypothesis requires furthermore $e$ to be a convex function of $s$ and $v$, which means:
\begin{align*}
 \gamma\delta-\Gamma^2&\geq 0\,, &\delta&\geq 0 \,,& \gamma&\geq 0\,.
\end{align*}
In particular,  $\gamma\geq 0$ means that $p$ increases with the density $\rho= 1/v$, which allows us to define the \emph{adiabatic sound speed} by 
\begin{equation}\label{eq:c}
c= \sqrt{\left.\frac{\pt p}{\pt \rho}\right|_s} =\sqrt{\gamma\frac{p}{\rho}}\,.
\end{equation}
Then, we show that $\mathscr{G}$   can be expressed in function of $\rho$ and $c$ through the expression $\displaystyle \mathscr{G}=\frac{1}{c}\left. \frac{\pt(\rho c)}{\pt \rho}\right|_s$. 

%

Furthermore, we require usually $\Gamma>0$ and $\mathscr{G}>0$.
 The  condition $\Gamma>0$  is not thermodynamically required but is satisfied for many gases and ensures that the isentropes do not cross each other in the $(v,p)$ plan. 
The condition $\mathscr{G}>0$ means that the isentropes are strictly convex in the $(v,p)$ plan.

\subsubsection{Van der Waals Gas}
\begin{definition}
A gas is said to follow the \emph{Van der Waals} law, if it satisfies the following pressure law:
\begin{equation}\label{eq:VdW}
p\left(v- b \right)=\mathfrak{R}T\,,
\end{equation}
where $v$ is the massic\footnote{also called \emph{specific}} volume and $b$ is the \emph{covolume}, representing the compressibility limit of the fluid, due to the volume of the molecules.
\end{definition}
The Van der Waals law is  a modification of  the perfect gas law, in which $b=0$. In opposition to the perfect gas law, it  takes into account  the proper size of the molecules, which is important in some situations when the gas is strongly compressed. In this model,  the density must be bounded and the maximal density is  $\rho_{max}=\frac{1}{b}$.   

The fundamental relationship (\ref{eq:fond}) gives us  the PDE: $\pt_v e+\frac{\mathfrak{R}}{v-b}\pt_s e=0$. 
Thus, we introduce new variables  $w=(v-b)^{-\mathfrak{R}}$, $\sigma=(v-b)^{-\mathfrak{R}}\exp(s)$ and $\hat e(w,\sigma)=e(v,s)$.  We obtain $\pt_w\hat e=0$, so that  $e=\mathcal{E}((v-b)^{-\mathfrak{R}}\exp(s))$ for any regular function  $\mathcal{E}$. 

If we assume furthermore that $c_v$ is  constant, thanks to the definition of $c_v$ and (\ref{eq:fond}), we get that $\left.\frac{\pt^2 e}{\pt s^2}\right|_v=\frac{1}{c_v}\left.\frac{\pt e}{\pt s}\right|_v$, hence $\sigma \mathcal{E}''=(\frac{1}{c_v}-1)\mathcal{E}'$ and  $\mathcal{E}(\sigma)=C\sigma^{1/c_v}$ which leads to:
\begin{align*}
e&=(v-b)^{-\frac{\mathfrak{R}}{c_v}}\exp(\frac{s}{c_v})\,,& p&= \frac{\mathfrak{R}}{c_v}\frac{e}{v-b}\,.
\end{align*}
Some computations allow us finally to obtain
\begin{align}
\gamma &=\gamma_0\frac{v}{v-b},& \Gamma=\delta&=(\gamma_0-1)\frac{v}{v-b},&\mathscr{G}&=\frac{\gamma_0+1}{2}\frac{v}{v-b}\,, \nonumber
\end{align}
where
\begin{equation}\label{eq:g}
\g=\frac{\mathfrak{R}}{c_v}+1\,.
\end{equation}
The conditions of Section \ref{sec:constraint}  are then satisfied for $\gamma_0>1$.

\begin{remark}
\begin{enumerate} 
\item  A perfect gas can be seen as a Van der Waals gas with $b=0$. A perfect gas for which $c_v$ is constant is called  \emph{polytropic}. 
\item  The Van der Waals law coincides with the \emph{dusty gas} law \cite{dustgret,JenaSharma,dustsaff,steiner,dustvish}.  In this model, we consider that the gas is perfect but polluted by dust particles that are equidistributed and have a non-negligible volume.
\end{enumerate}
\end{remark}

Very often in the literature, the perfect polytropic gases are considered as a canonical example. However, their adequation with physical observations is not as good as for Van der Waals gases, for example in explosion phenomena, or in the sonoluminescence phenomenon \cite{sonolum, Evans, Lauterborn}.
%

In the following, we only consider Van der Waals fluids with constant  and strictly positive calorific capacity $c_v$:
\begin{equation}\label{eq:gpos}
c_v>0\,,
\end{equation}
which implies $\g :=\frac{\mathfrak{R}}{c_v}+1>1$.

\subsection{Symmetrisation}\label{sec:sym}
An important property of the Euler equation is their symmetrisability.
\subsubsection{General case, without vacuum}\label{sec:sym_ssvide}
\noindent If $\rho>0$ and $\left.\frac{\pt p}{\pt \rho}\right|_s>0$ then the system (\ref{eq:euler}) can be written in the variables $(p,\rho, s)$. Then, the system is almost symmetric, since it can be written matricially $\pt_t \tilde V +\sum_k \tilde A_k(\tilde V) \pt_k \tilde V=0$, with $\tilde V=(p,u^{\mathbf{T}},s)^{\mathbf{T}}$ and
\[
\tilde A(\xi, V)= \sum_k\xi_k \tilde A_k (\tilde V)=\left(
\begin{array}{ccc}
u\cdot \xi&\rho c^2 \,\xi^{\mathbf{T}}&0\\
\frac{1}{\rho}\xi&u\cdot \xi I_d&0\\
0&0&u\cdot \xi
\end{array}
\right)\,.
\]

This matrix is almost symmetric since we obtain a symmetric matrix by multiplying it on the left by $D:=\mathrm{Diag}\left(\frac{1}{\rho c^2}, \rho,\ldots,\rho,1\right)$. Consequently, we have  the following
\begin{proposition}\label{prop:sym1}
The system (\ref{eq:euler}) is Friedrichs symmetrisable when  $(\rho,u,s)$  takes values in a compact subset of $\rpis\times\reali^d\times\reali$.
\end{proposition}
 Indeed, for such values of   $(\rho,u,s)$,  $D\tilde A(\xi,V)$ is symmetric  with $D$ symmetric definite positive.

\subsubsection{Van der Waals gas}\label{sec:symMUK}
We are not completely satisfied with the previous formulation as it does not authorise $\rho$ to vanish, or even to tend to 0 at infinity. For example, $\rho$ cannot be taken into $\H{m}$ for $m\geq 0$.
Makino et al.   have introduced in \cite{MakinoUK}  a symmetrisation for perfect polytropic gases allowing the  null density areas. We generalise here their method to the case of Van der Waals gases.

First, let us  remind that for Van der Waals gases, we have  $\gamma=\frac{\g}{1-b\rho}$ and $p=(\g-1)\left(\frac{\rho}{1-b\rho}\right)^{\g}\exp\left(\frac{s}{c_v}\right)$. We now introduce the new variable
\[
 \pi =2\sqrt{\frac{\g}{\g-1}} \left(\frac{p}{\g-1}\right)^{\frac{\g-1}{2\g}}\,,
\]
and we re-write the system (\ref{eq:euler}) in the variables $(\pi, u,s)$. In order to do that, we first write the system (\ref{eq:euler}) in $(p,u,s)$ variables:
\begin{equation}\label{eq:presq_sym}
\left\{
\begin{array}{l}
\pt_t p + u\cdot \nabla p +\rho c^2 \div u=0 \,,\\
\pt_t  u + (u\cdot\nabla )u + \frac{1}{\rho}\nabla p=0\,,  \\
\pt_t s +
u\cdot\nabla s=0\,.
\end{array}
\right.
\end{equation}

Since $\pi=f(p)=Cp^\alpha$,  it is sufficient to multiply the first line by  $f'(p)=C\alpha p^{\alpha-1}$  to obtain an equation in $\pi$:
\[
\left\{
\begin{array}{rcl}
\pt_t \pi+u\cdot \nabla \pi +C\alpha p^{\alpha-1}\rho c^2  \nabla\cdot u &=&0\,,\\
\pt_t u +(u\cdot \nabla)u +\frac{1}{\rho C\alpha p^{\alpha-1}} \exp(\frac{s}{\g c_v})\nabla \pi &=&0\,,\\
\pt_t s +u\cdot \nabla s&=&0\,.
\end{array}
\right.
\]
Besides, we know that $c^2=\frac{\g}{1-b\rho}\frac{p}{\rho}$, $C=2\sqrt{\frac{\g}{\g-1}}(\frac{1}{\g-1})^{\frac{\g-1}{2\g}}$ and $\alpha=\frac{\g-1}{2\g}$. It remains to evaluate the coefficients:
\begin{align*}
C\alpha p^{\alpha-1}\rho c^2&=\frac{\g \alpha}{1-b\rho}\pi\\ 
&=\frac{\g-1}{2}\frac{\pi}{1-b\rho}\,,\\
\frac{1}{\rho C\alpha p^{\alpha-1}}&=\frac{\exp(s/(\g c_v)) }{1-b\rho}\frac{(\g-1)^{1/\g}}{\alpha C^{(\g-1)/(\g \alpha)}}\pi^{\frac{1-\alpha-1/\g}{\alpha}}\\
&=\exp\left(\frac{s}{\g c_v}\right)\frac{\g-1}{2}\frac{\pi}{1-b\rho}\,.
\end{align*}
Thus,
\begin{equation}\label{eq:symMUK}
\left\{
\begin{array}{rcl}
\pt_t \pi+u\cdot \nabla \pi +\frac{\g-1}{2}\frac{\pi }{1-b\rho} \nabla\cdot u &=&0\,,\\
\pt_t u +(u\cdot \nabla)u +\frac{\g-1}{2}\frac{\pi}{1-b\rho} \exp(\frac{s}{\g c_v})\nabla \pi &=&0\,,\\
\pt_t s +u\cdot \nabla s&=&0\,.
\end{array}
\right.
\end{equation}
Moreover, $1/(1-b \rho)=1+b\left(\frac{\g-1}{4\g}\right)^{\frac{1}{\g-1}} \exp\left(\frac{-s}{\g c_v}\right)\pi^{\frac{2}{\g-1}}$.  Therefore, denoting
\begin{equation}\label{eq:exposant}
\e=\frac{\g+1}{\g-1}>1\,, \qquad \textrm{ and  }\qquad \tilde b= b\left(\frac{\g-1}{4\g}\right)^{\frac{1}{\g-1}}\,,
\end{equation}
we obtain $ \frac{\pi}{1-b\rho}= \pi(1+\tilde b e^{-s/(\g c_v)} \pi^{\e-1})$, which is well defined for all  $\pi\geq 0$ and in particular for $p=0$, since we have assumed  in (\ref{eq:gpos})  $\g>1$.  The matrix associated to the system  (\ref{eq:symMUK}) writes now 
\[
A(\xi,\pi,u,s)=
\left(
\begin{array}{ccc}
u\cdot\xi & \frac{\g-1}{2}\frac{\pi}{1-b\rho} \, \xi^{\mathbf{T}}& 0\\
\frac{\g-1}{2}\frac{\pi}{1-b\rho} \exp(\frac{s}{\g c_v})\xi& u\cdot \xi I_d&0\\
0&0&u\cdot \xi
\end{array}
\right)\,.
\]
It is once again ``almost symmetric'' in the sense that $SA$ is symmetric, $S$ being the  diagonal definite positive matrix
\[
S= \mathrm{Diag}\left(1, \exp(-s/(\g c_v)),\ldots, \exp(-s/(\g c_v)), 1 \right)\,.
\]
Furthermore, this symmetriser is independent from $p$ and in particular is well defined and definite positive even when $p$ or $\rho$ vanishes. Finally, we have  the proposition:
\begin{proposition}\label{prop:sym2}
For a  Van der Waals gas with constant $c_v$ and $c_v>0$,  the system of Euler equations can be written for regular solutions  as  (\ref{eq:symMUK}) which is Friedrichs symmetrisable for $(p,u,s)\in \mathcal{K}$, where $\mathcal{K}$  is a compact subset of  $\rpic\times\reali^d\times\reali$.
\end{proposition}

Note that in the variables $(\pi, u, s)$,  the system  (\ref{eq:symMUK}) can be written
\begin{equation}\label{eq:symMUK2}
\left\{
\begin{array}{rcl}
\pt_t \pi+u\cdot \nabla \pi +\frac{\g-1}{2} (1+\tilde b \exp(\frac{-s}{\g c_v}) \pi ^{\e-1}) \pi \nabla\cdot u &=&0\,,\\
\pt_t u +(u\cdot \nabla)u +\frac{\g-1}{2}\exp(\frac{s}{\g c_v}) (1+\tilde b\exp({\frac{-s}{\g c_v}}) \pi ^{\e-1}) \pi\nabla \pi &=&0\,,\\
\pt_t s +u\cdot \nabla s&=&0\,,
\end{array}
\right.
\end{equation}
where $\nu $ and $\tilde b$  are defined as in (\ref{eq:exposant}).

\subsection{Local existence}

The symmetrisation of Proposition \ref{prop:sym2} is very useful to show local existence of regular solutions with vanishing density.

\begin{theorem}\label{thm:ex_loc_sym}
We consider a Van der Waals gas with constant $c_v$ such that  $\g\in ]1,3]$. Let
\[
 \pi_0=2\sqrt{\frac{\g}{\g-1}}\left(\frac{\rho_0}{1-b\rho_0}\right)^{(\g-1)/2} \exp(\frac{\g-1}{2\g c_v} s_0)\,,
\]
 where  $\rho_0\in \C1(\reali^3;[0,1/b[)$.  We introduce also $\nu=\frac{\g+1}{\g-1}$, $\tilde b=b\left(\frac{\g-1}{4\g}\right)^{\frac{1}{\g-1}}$. We assume that $(\pi_0, u_0,s_0)\in\H{m}(\reali^d)$ for $m>1+\frac{d}{2}$. Then there exists $T>0$ and a unique solution $(\pi,u,s)\in \C1([0,T]\times \reali^3)$  to the Cauchy problem for
 \begin{equation}\label{eq:symMUK3}
\left\{
\begin{array}{rcl}
\pt_t \pi+u\cdot \nabla \pi +\frac{\g-1}{2} (1+\tilde b \pi ^{\e-1}) \pi \nabla\cdot u &=&0\,,\\
\pt_t u +(u\cdot \nabla)u +\frac{\g-1}{2}\exp(\frac{s}{\g c_v}) (1+\tilde b \exp(\frac{-s}{\g c_v})\pi ^{\e-1}) \pi\nabla \pi &=&0\,,\\
\pt_t s +u\cdot \nabla s&=&0\,,
\end{array}
\right.
\end{equation} 
with initial condition $(\pi_0,u_0,s_0)$.
Furthermore, 
\[
(\pi, u,s)\in\C{}([0,T];\H{m}(\reali^d;\reali^{d+2}))\cap\C1([0,T];\H{m-1}(\reali^d;\reali^{d+2}))\,.
\]  
\end{theorem}

\subsection{Positivity of the density}
For regular solutions,  the positivity of the density  is given by the following
\begin{proposition}\label{prop:rhopos}
Let $(\bar \rho,\bar u ,\bar s )\in \C1([0,T_{\mathrm{ex}}[\times \reali^d)$ be  a regular solution  of the Cauchy problem  (\ref{eq:euler}) associated to the regular initial conditions $(\rho_0,u_0,s_0)\in \C1(\reali^d) $. If $\nabla \bar u \in \L\infty([0,T]\times \reali^d)$ for all $T<T_{\mathrm{ex}}$, $\rho_0\in \L\infty(\reali^d,\reali)$ and if $\rho_0(x)\geq 0$ for all $x\in \reali^d$, then for all $t\in [0,T_{\mathrm{ex}}[$,
\[
\rho(t,x)\geq  0\,.
\]
\end{proposition}

\begin{proof}
We use the characteristics method  to obtain an expression of the solution of the Cauchy problem (\ref{eq:app}). 
Assuming $\nabla\bar u $ is bounded, we obtain:
\[
\rho(t,x)=\rho_0(X(0;t,x))\exp\left(-\int_0^t \div \bar u(\tau, X(\tau; t,x))\d{\tau} \right)> 0\,,
\]
where $X$  is a solution of the Cauchy problem
\[
\frac{\d{X}}{\d{t}}=\bar u(t,X)\,,\qquad\qquad X(t_0;t_0,x_0)=x_0\,,
\]
which is global in time since $\nabla\bar u $ is  bounded.
\end{proof}

We also prove that for a Van der Waals gas with constant $c_v>0$, the  variable $\pi=2\sqrt{\frac{\g}{\g-1}} \left(\frac{p}{\g-1}\right)^{\frac{\g-1}{2\g}}$ introduced in section \ref{sec:symMUK}  remains  non-negative if it is non-negative  at initial time. This property implies in particular that if   $\rho_0<1/b$,  then, as long as the regular solution exists, this property will be satisfied.
\begin{proposition} \label{prop:rholeqb}
We consider a Van der Waals gas with constant $c_v$. We denote $\pi=2\sqrt{\frac{\g}{\g-1}} \left(\frac{p}{\g-1}\right)^{\frac{\g-1}{2\g}}$, $\nu=\frac{\g+1}{\g-1}$ and $\tilde b=b\left(\frac{\g-1}{4\g}\right)^{\frac{1}{\g-1}}$. 
Let $(\bar \pi,\bar u ,\bar s )\in \C1([0,T_{\mathrm{ex}}[\times \reali^d)$ be  a regular solution of  
(\ref{eq:symMUK2}) 
satisfying the initial conditions
\begin{align*}
\pi(0,x)&=\pi_0(x)\,,&u(0,x)&=u_0(x)\,,&s(0,x)=s_0(x)\,,
\end{align*}
with $(\pi_0,u_0,s_0)\in \C1(\reali^d) $ and $s_0\in \L\infty$. If $\div \bar u \in \L\infty([0,T]\times \reali^d)$  for all  $T<T_{\mathrm{ex}}$,  $\pi_0\in \L\infty(\reali^d,\reali)$ and  $0\leq \rho_0(x)<1/b$ for all $x\in \reali^d$, then for all $t\in [0,T_{\mathrm{ex}}[$, $\pi\geq 0$.  Then, we can define $\rho$, and we have: 
\[
0\leq \rho(t,x)<1/b\,.
\]
\end{proposition}

\begin{proof} Let $T_0<T_{\mathrm{ex}}$.
We introduce the Cauchy problem
\begin{align}
\pt_t w+\div (w\bar u)&=g(t,x,w)\,,\label{scalaire}\\ 
 w(0,x)&=w_0(x)\,,\nonumber
\end{align}
where
\[
g(t,x,w)=\left( 1-\frac{\g-1}{2}(1+\tilde b \exp(\frac{-\bar s(t,x)}{\g c_v}) {w}^{\e-1}) \right)  w \div\bar u(t,x)
\]
We can apply the  Kru\v zkov theorem \cite{Kruzkov, Sr1}. Indeed, the hypotheses ensure that $g(t,x,w)-w\div (\bar u(t,x))=-\frac{\g-1}{2}(1+\tilde b e^{-\bar s/(\g c_v)}{w}^{\e-1})w \div\bar u$ is uniformly bounded with respect to  $x\in \reali^d$ when $w$ is considered as a variable taking values in a compact set. Furthermore, $\pi_0\in \L\infty$,  so the regular solution  $\bar \pi$ coincides with the entropy solution $w_1$ of (\ref{scalaire}) associated to $w_{0,1}=\pi_0$.

Besides, the entropy solution $w_2$ of  (\ref{scalaire}) associated  to the initial condition  $w_{0,2}\equiv 0$ is the  function constantly equal to 0. 

After Kru\v zkov Theorem  $w_{0,1}\geq w_{0,2}$ implies $w_1\geq w_2$ for all $(t,x)\in [0,T_0]\times \reali^d$, that is to say $\bar \pi(t,x) \geq 0$ for all $(t,x)\in [0,T_0]\times \reali^d$. 
The formula:
\[
\rho=\frac{1}{b}\left(1-\frac{1}{1+\tilde b \exp(\frac{-s}{\g c_v}) \pi^{\frac{2}{\g-1}}}\right)
\]
allows us to conclude.
\end{proof}

\section{Proof of Theorem \ref{thm:VdW}}\label{sec:proof}

In order to prove this theorem, we adapt M. Grassin's idea \cite{Grass}. First, we  look to the isentropic case, which allows us to simplify the estimates. For a Van der Waals gas,   non-linear terms now  appear in  the estimate which we need to treat separately.

\subsection{Isentropic case}\label{sec:isen}
Let us consider first the isentropic case
 \begin{equation}\label{eq:symMUK2ise}
\left\{
\begin{array}{rcl}
\pt_t \pi+u\cdot \nabla \pi +\frac{\g-1}{2} (1+\tilde b \pi ^{\e-1}) \pi \nabla\cdot u &=&0\,,\\
\pt_t u +(u\cdot \nabla)u +\frac{\g-1}{2}(1+\tilde b \pi ^{\e-1}) \pi\nabla \pi &=&0\,,\\
\end{array}
\right.
\end{equation}
with initial conditions 
\begin{align}
\pi(0,x)&=\pi_0(x)\,,& u(0,x)&=u_0(x)\,.\label{eq:inisen}
\end{align}
which is technically simpler than the general case, but provides estimates very useful in order to treat the general case.

We consider besides the problem 
 \begin{equation}\label{eq:app2}
 \left\{\begin{array}{cccl}
 \pt_t \bar{u} +(\bar{u}\cdot \nabla)\bar{u}&=&0 & \textrm{on } \reali_+\times \reali^d\,,\\
         \bar{u}(0,x)&=&u_0(x)& \textrm{on } \reali^d\,,
     \end{array}\right.
 \end{equation}
 obtained by neglecting $\pi$ in (\ref{eq:symMUK2ise}). After  \cite[Lemme 3.1 and Prop. 3.1]{Grass}  we have the following preliminary result
\begin{proposition}\label{prop:ubar}
Under hypotheses \textbf{(H2)} and \textbf{(H3)},  the problem (\ref{eq:app2}) admits a global regular solution  $\overline u$ satisfying
\begin{enumerate}
\item $\D{\overline u} (t,x)=\frac{\Id}{1+t}+\frac{K(t,x)}{(1+t)^2}$, for all $x\in \reali^d$ and $t\in \rpic$,
\item $\norma{\D{}^l \overline u (t,\cdot)}_{\L2}\leq K_l(1+t)^{d/2-(l+1)}$ for $l\in \lb 2,M+1 \rb$,
\item $\norma{\D{}^2\overline u(t,\cdot)}_{\L\infty}\leq C(1+t)^{-3}$,
\end{enumerate}
where $K:\rpic\times\reali^d\to \mathcal{M}_d(\reali)$, $\norma{K}_{\L\infty(\rpic\times \reali^d)}\leq K_1$,  $C$ and $K_l$ for $l\in \lb 1,m+1\rb$  being non-negative constants not depending on $m,d,\delta, \norma{u_0}_X$.
\end{proposition}

\subsubsection{Local uniqueness}\label{sec:uniloc_ise}
\begin{proposition}\label{prop:uni_loc}
Let  $U_0=(\pi_0,u_0)^{\mathbf{T}}\in \H{m}(\reali^d)$ and $\tilde U_0$ be two initial data for (\ref{eq:symMUK2ise}). Let  $U=(\pi,u)^{\mathbf{T}}$, $\tilde U$ be two corresponding solutions, defined for $0\leq  t \leq T_0$. We assume that $\norma{\D{\tilde U}}_{\L\infty([0,T_0]\times \reali^d)}<\infty$. Let $x_0\in \reali^d$ and $R \geq 0$. We denote 
\begin{align}
M&=\sup \{ (\frac{\g-1}{2} \modulo{\pi}(1+ \tilde b \pi^{\e-1})+\modulo{u})(t,x)\,, (t,x)\in [0,T_0]\times B(x_0,R) \}\,,\label{eq:M}\\
C_T&=\{
(t,x)\in [0,T]\times B(x_0,R-Mt)
\} \qquad\textrm{ for }T\in[0,T_1]\,,\label{eq:ct}
\end{align}
where $T_1=\min(T_0,\eta/M)$.
If $U_0=\tilde U_0$ on $B(x_0,R)$ then $U=\tilde U$ on $C_{T_1}$.
\end{proposition}
The proof of this proposition is classical for hyperbolic systems (see for example \cite{Racke}), the constant $M$ being the maximal propagation speed.

\begin{proof}
Let $U_0=(\pi_0,u_0), \tilde U_0=(\tilde\pi_0,\tilde u_0)$ be two initial data  for (\ref{eq:symMUK2ise})  such that $U_0\in \H{m}$. Let  $U,\tilde U$  be the solutions of the associated Cauchy problems. We assume that these solutions are defined on  $[0,T_0]$ with $T_0>0$. 
Let also  $x_0\in \reali^d$, $\eta\in \rpis$ and $M$, $C_T$  be as in (\ref{eq:M}) and (\ref{eq:ct}). Then we have: 
\[
\pt_t U +\sum_j a_j(U)\pt_j U=0\,,
\]
where
\[
 a_j(U)=\left(
\begin{array}{cccc}
u_j&0& \frac{\g-1}{2}(\pi+\tilde b \pi^\e)&\\
0&\ddots&&\\
\frac{\g-1}{2}(\pi +\tilde b\pi^\e)&&&0\\
\vdots&&\ddots&0\\
0&0&&u_j
\end{array}
\right)\,.
\]
Consequently, $\pt_t (U- \tilde U) +\sum_j a_j(U)\pt_j(U-\tilde U)+ (a_j(U)-a_j(\tilde U))\pt_j \tilde U=0$. Then, we make the scalar product with $(U-\tilde U)$ and we integrate on  $C_T$ for $T\in [0,T_0]$. We get
\begin{align*}
&\frac{1}{2}\int_{C_T} \pt_t \modulo{U-\tilde U}^2 +\sum_j \pt_j\left((U-\tilde U)\cdot a_j(U)(U-\tilde U)\right)\\ 
&\qquad  -\sum_j (U-\tilde U)\cdot \pt_j(a_j(U))(U-\tilde U) \d{x}\d{t}
=-\int_{C_T} (U-\tilde U)\cdot (a_j(U)-a_j(\tilde U))\pt_j \tilde U\d{x}\d{t}\,.
\end{align*}
Then using the Stokes formula and noting that $\pt C_T=(\{0\}\times B(x_0,\eta))\cup (\{T\}\times B(x_0,\eta-MT))\cup \mathcal{C} $, we obtain
\begin{align*}
&\frac{1}{2}\int_{B(x_0,R-MT)}\modulo{U-\tilde U}^2(T,x)\d{x} -\frac{1}{2}\int_{B(x_0,\eta)}\modulo{U-\tilde U}^2(0,x)\d{x}\\
&+\frac{1}{2\sqrt{1+1/M^2}}\int_{\mathcal{C}} \modulo{U-\tilde U}^2 + \sum_j (U-\tilde U)\cdot a_j(U)(U-\tilde U)\frac{x_j}{M\modulo{x}}\d{\sigma}\\
=&\int_{C_T}   \sum_j (U-\tilde U)\cdot \pt_j(a_j(U))(U-\tilde U) -(U-\tilde U)\cdot (a_j(U)-a_j(\tilde U))\pt_j \tilde U\d{x}\d{t}\,.
\end{align*}
Besides $\norma{\pt_j a_j(u)}_{\L\infty} \leq C\norma{\D{U}}_{\L\infty(C_T)}(1+\norma{\pi}_{\L\infty(C_T)}^{\e-1})$. Hence,
\begin{align*}
&\int_{C_T}  (U-\tilde U)\cdot \pt_j(a_j(U))(U-\tilde U)-(U-\tilde U)\cdot (a_j(U)-a_j(\tilde U))\pt_j \tilde U\d{x}\d{t}\\
\leq& C\norma{\D{U}}_{\L\infty(C_T)}(1+\norma{\pi}_{\L\infty(C_T)}^{\e-1})\int_0^T \int_{B(x_0,R-Mt)}\modulo{U-\tilde U}^2\d{x}\d{t}\,.
\end{align*}
Furthermore, the choice of  $M$ implies 
\[
 \int_{\mathcal{C}} \modulo{U-\tilde U}^2 + \sum_j (U-\tilde U)\cdot a_j(U)U-\tilde U)\frac{x_j}{M\modulo{x}}\d{\sigma}\geq 0\,,
\]
so finally we get the estimate
\begin{align*}
&\frac{1}{2}\int_{B(x_0,R-MT)}\modulo{U-\tilde U}^2(T,x)\d{x} -\frac{1}{2}\int_{B(x_0,\eta)}\modulo{U-\tilde U}^2(0,x)\d{x}\\
\leq &  C\norma{\D{U}}_{\L\infty(C_T)}(1+\norma{\pi}_{\L\infty(C_T)}^{\e-1})\int_0^T \int_{B(x_0,R-Mt)}\modulo{U-\tilde U}^2\d{x}\d{t}\,.
\end{align*}
We conclude thanks to Gronwall lemma that
\[
\frac{1}{2}\int_{B(x_0,R-MT)}\modulo{U-\tilde U}^2(T,x)\d{x}\leq\frac{1}{2}e^{C'T} \int_{B(x_0,R)}\modulo{U_0-\tilde U_0}^2(x)\d{x}
\]
where $C'=C\norma{\D{U}}_{\L\infty(C_T)}(1+\norma{\pi}_{\L\infty(C_T)}^{\e-1})$.
\end{proof}

\subsubsection{Local Existence}
We construct a local solution of  (\ref{eq:symMUK2ise})--(\ref{eq:inisen}) such that the difference between this solution  and  $(0,\overline u)$ be in $\C0(\rpic;\H{m}(\reali^d;\reali^{d+1}))\cap \C1(\rpic;\H{m-1}(\reali^d;\reali^{d+1}))$. 
The first step is the symmetrisation of the system, given by Proposition \ref{prop:sym2}.  This result allows us to use a general theorem (see Theorem \ref{thm:ex_loc_sym}) giving the local existence of solution. Let us define as above, 
\begin{equation}\label{eq:piisen}
\pi =2\sqrt{\frac{\g }{\g-1}}\left(\frac{p}{\g-1}\right)^{\frac{\g -1}{2\g}}=2\sqrt{\frac{\g }{\g-1}}\left( \frac{\rho}{1-b\rho} \right)^{\frac{\g-1}{2}}\,,
\end{equation}
where we assume $0\leq \rho< 1/b$. Then we use the same  method as M. Grassin \cite{Grass} to prove that the system (\ref{eq:symMUK2ise}) admits a local in time solution, with initial condition $u_0$ in the space $X$ and not in a Sobolev space (in particular, $u_0$ does not tend to 0 at infinity). We use here the compactness of the support of $\rho_0$ (hypothesis \textbf{(H4)})  and the finite propagation speed of the solutions for an hyperbolic system. More precisely, we assume that $\mathrm{Supp}(\rho_0)\subset B(0,R)$ for $R>0$. 
Let  $\eta>0$ and $\phi\in \Cc\infty(\reali^d;\rpic)$  be such that $\phi\equiv 1$ on $B(0, R+2\eta)$. 
We obtain a local in time solution  $(\pi^\phi,u^\phi)$ of the problem (\ref{eq:symMUK2ise}) with initial conditions $(\pi_0,\phi u_0)\in\H{m}$ for $t\in [0,T[$.   The Propositions \ref{prop:rhopos} and \ref{prop:rholeqb} ensure that the condition $0\leq \rho<1/b$
is satisfied.

Let $\varepsilon \in ]0, T[$ ,  we introduce  the \emph{maximal propagation speed}  $M=\sup \{\frac{\g-1}{2}(\modulo{\pi^\phi }+\modulo{\pi^\phi}^\e) +\modulo{u^\phi}\,;\; t\in [0,T-\varepsilon ]\,,\; x\in \reali^d \} $. We also introduce  $\varepsilon' \in ]0,\frac{\eta}{2M}[$ and  $T_1=\min(T-\varepsilon, \frac{\eta}{2M}-\varepsilon')$ the time for which this construction is available. We finally obtain a solution $(\pi,u)$ of (\ref{eq:symMUK2ise})--(\ref{eq:inisen}) by denoting 
\[
(\pi,u)=\left\{
\begin{array}{l}
(\pi^\phi,u^\phi) \textrm{ in } \mathcal{K}\,,\\
(0,\overline u) \textrm{ out of  } \mathcal{K}\,,
\end{array}
\right.
\]
where $\mathcal{K}$ is the cone $\mathcal{K}=\{(t,x)\, ; \; 0\leq t\leq T_1\,,\; x\in B(0,R+\eta+Mt) \}$.  Then it is sufficient to show that the solutions can be glued smoothly along $\pt \mathcal{K}$. We use here the property of local uniqueness given below by Proposition \ref{prop:uni_loc}. Let indeed $x_0\in S(0,R+\eta)$ be the sphere of radius $R+\eta$ of centre 0 and  $E_{x_0}=\{(t,x)\,;\; t\in [0,T_1], x\in B(x_0,\eta-Mt)\}$. The choice of  $T_1$ implies in particular $\pt \mathcal{K}\subset \cup_{x_0\in S(0,R+\eta)} E_{x_0} $ (see Fig. \ref{fig:cone}).
\begin{figure}[ht] 
\begin{center} 
\includegraphics[width=7.5cm]{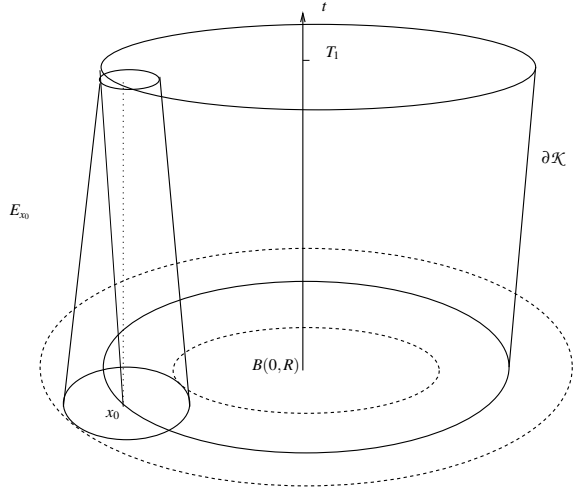} 
\caption{Gluing of the solutions along $\partial\mathcal{K}$.}\label{fig:cone}
\end{center}
\end{figure}
The initial conditions of  $(\pi^\phi,u^\phi)$ and $(0,\overline u)$ coincide on $E_{x_0}$  since the support of  $\pi_0$  is included in $B(0,R)$.  Proposition \ref{prop:uni_loc}  allows us then to claim that  $(\pi^\phi,u^\phi)=(0,\overline u)$ on $E_{x_0}$ and consequently on a neighbourhood of  $\pt\mathcal{K}$.

\subsubsection{Energy Estimates}
When we compare to the Perfect Polytropic Gas case \cite{Grass},  we observe that the system (\ref{eq:symMUK2ise}) has an additional term, which will modify the estimates. 

Let us  denote $U=(\pi, u-\overline u)$, $w=u-\overline u$ and $\overline U=(0,\overline u)$. We also introduce
\begin{align}
A_j(U)&=\left(
\begin{array}{ccccc}
u_j-\bar u_j&0&\ldots& \frac{\g-1}{2} \pi&0\\
0&u_j-\bar u_j&0&&0\\
\frac{\g-1}{2} \pi &0&\ddots&&\vdots\\
&\vdots&&&0\\
0&0&\ldots&0&u_j-\bar u_j
\end{array}
\right)\,,&
B(\D{\overline U},U)&=\left(
\begin{array}{c}
\frac{\g-1}{2} \pi \div \overline u\\
(w\cdot \nabla)\overline u
\end{array}
\right)\,,\label{eq:AB}
\end{align}
so that the system  \Ref{eq:symMUK2ise} can be written 
\begin{equation}\label{eq:matrice}
\pt_t U+\sum_{j=1}^d A_j(U)\pt_j U=-B(\D{ \overline U}, U)-\sum_{j=1}^d \overline u_j \pt_j U- F(\D{ \overline U},\D{U}, U) 
\end{equation}
where
\begin{equation}\label{eq:F}
F(\D{ \overline U},\D{U}, U)=\frac{\g-1}{2} \tilde b\pi^{\e} 
\left( 
\begin{array}{c}
\div (w+\overline u)\\
\nabla \pi
\end{array}
\right)\,,\qquad\textrm{ with } \quad \e=\frac{\g+1}{\g-1}\,.
\end{equation}

Observing the properties of  $\overline u$ described in  Proposition \ref{prop:ubar}, we expect the terms $\norma{\D{}^k U}_{\L2}$ for $k\in \lb0,m\rb $  to decrease with respect to time with a rate depending on  $k$.  Consequently, we introduce 
\begin{align}
Y_k(t)&=\norma{\D{}^k U(t,\cdot)}_{\L2}\,,&
Z(t)&=\sum_{k=0}^m (1+t)^{g_k}Y_k(t)\,,\label{eq:normez}
\end{align}
with $g_k=k+ c$, in which $c$  has to be chosen so that all the terms of  $Z$  have the same decreasing in time.
In order to estimate  $Z$,  we apply the operator $\D{}^k$ to \Ref{eq:matrice}, we make the scalar product with  $\D{}^k U$ and we integrate on  $\reali^d$. The system (\ref{eq:matrice}) is different from the one considered by  M.~Grassin through the term $F(\D{\overline U}, \D{U}, U)$ defined in \Ref{eq:F}. We use now  \cite[Prop. 3.2 and 3.3]{Grass} to estimate the terms in common. We remind these results here:
\begin{proposition}\label{prop:rs}
Let $\alpha\in\naturali^d$ be  a $d$-uplet of size $k\geq 0$ (that is to say $\modulo{\alpha}_1 = \alpha_1+\ldots+\alpha_d=k$).  Let us denote, for  $U=(\pi,u-\bar u)$,
\begin{align*}
R_k&=-\int_{\reali^d}\D{}^k U\cdot\D{}^k\left(\sum_{j=1}^d A_j(U)\pt_j U\right)\d{x}\,,\\ 
S_k&=-\int_{\reali^d}\D{}^k U\cdot\D{}^k\left(B(\D{ \overline U}, U)+\sum_{j=1}^d \overline u_j \pt_j U\right)\d{x}\,,
\end{align*} 
then there exists a constant $C>0$ depending only on  $m,d$ such that
\begin{align*}
\modulo{R_k}&\leq C (1+t)^{-g_1-d/2}Y_k^2 Z\,,&
S_k &\leq C(1+t)^{-g_k-2}Y_k Z-\frac{k+r}{1+t}Y_k^2\,,
\end{align*}
where
\begin{align}
r&=\min(1-\frac{d}{2},(\frac{\g}{2}-1) d)\,.\label{eq:rr}
\end{align}
\end{proposition}
We have now to estimate $-\int_{\reali^d} \D{}^k U\cdot \D{}^k(F(\D{\overline U}, \D{U}, U))$.
Let us  denote  $\check b=\frac{\g-1}{2}\tilde b$ and
\begin{align*}
 I&=-\check b \int_{\reali^d} \D{}^k U\cdot \D{}^k\big(\pi^{\e} 
\big(\begin{array}{c}
\div w\\ 
\nabla \pi
\end{array}
\big)\big),
& J&=-\check b \int_{\reali^d} \D{}^k U\cdot \D{}^k\big(\pi^{\e} 
\big(\begin{array}{c}
\div \overline u\\  0
\end{array}
\big)\big)
\end{align*}
so that $-\int_{\reali^d} \D{}^k U\cdot \D{}^k(F(\D{\overline U}, \D{U}, U))=I+J$. A priori, $J$  is easier to estimate than $I$. However the estimate of $I$ is possible since the matrix $\pi^\e\left(\begin{array}{cc} 0 &\xi^{\mathbf{T}}\\ \xi& 0\end{array}\right)$ is symmetric.

\paragraph{Estimate of $I$.} We show in Section \ref{sec:lem} the following
\begin{lemma}\label{lem:I} 
With the notations introduced in Section  \ref{sec:isen}
\begin{equation}\label{eq:Iise}
\modulo{I}\leq C \norma{\pi}_{\L\infty}^{\e-1}\norma{\D{U}}_{\L\infty}\norma{\D{}^k U}_{\L2}^2\,.
\end{equation}
\end{lemma}

\paragraph{Estimate of  $J$. }
Here, we divide $J$ in two parts: a first part $J_1$ which contains only first order derivatives of $\overline u$, and a second part  $J_2$ in which all the derivatives of  $\overline u$ are at least of order 2. More precisely,
\begin{align*}
J_1&=-\check  b \int_{\reali^d }\sum_{\alpha\in \naturali^d \mid \modulo{\alpha}_1=k}(\pt^\alpha \pi)( \pt^\alpha \pi^\e)\div \overline u \,,\\
J_2&= J-J_1\,.
\end{align*}
For $J_1$,  we use the first point of Proposition \ref{prop:ubar} giving the decreasing in time of  $\D{ \overline u}$,  and  Lemma \ref{lem:fp} giving the estimate:
\[
\norma{\pt^\alpha(\pi^\e)}_{\L2}\leq \norma{\pi}^{\e-1}_{\L\infty}\norma{\D{}^k\pi }_{\L2}\,,
\] 
for all $d$-uplet $\alpha\in \naturali^d$ of size $k$, that is to say satisfying $\modulo{\alpha}_1=\sum \alpha_i=k$.\\

We obtain, using the Cauchy-Schwarz inequality:
\begin{align}
\modulo{J_1}&\leq \check b \norma{\pt^\alpha \pi}_{\L2}\norma{\pt^\alpha (\pi^\e)}_{\L2} \norma{\div \overline u}_{\L\infty}\nonumber\\
&\leq \frac{C}{1+t}\norma{\D{}^k U}_{\L2}^2\norma{\pi}_{\L\infty}^{\e-1}\,.\label{eq:J1}
\end{align}


For $J_2$,  we prove in Section \ref{sec:lem} the following 
\begin{lemma}\label{lem:J2} 
With the notations introduced in Section  \ref{sec:isen},  there exists a constant $C>0$ such that
\begin{equation}\label{eq:J2}
\modulo{J_2}\leq C(1+t)^{d_k} \norma{\D{}^k U}_{\L2} Z^\e\,,
\end{equation}
where $d_k=(-g_1-\frac{d}{2}+1)(\e-1)-g_k-2$.
\end{lemma}

\paragraph{Re-assembling of the estimates. }
Assembling the results of  Proposition \ref{prop:rs}, the estimates (\ref{eq:Iise})--(\ref{eq:J1})--(\ref{eq:J2}) of $I$ and $J$, and finally using Lemma \ref{lem:Z}, we obtain
\begin{align}
\frac{1}{2}\frac{\d{Y_k^2}}{\d{t}}+\frac{k+r}{1+t}Y_k^2 & 
\leq C(1+t)^\beta Y_k^2 Z +C' (1+t)^{-g_k-2}Y_k Z +C (1+t)^{\beta+(\beta+1)(\e-1)}Y_k^2 Z^\e \nonumber\\
& +C(1+t)^{(\beta+1)(\e-1)-1}Y_k^2 Z^{\e-1} +C(1+t)^{(\beta+1)(\e-1)-g_k-2}Y_k  Z^\e\,,\label{eq:yk}
\end{align}
where we have denoted
\[
\beta=-g_1-\frac{d}{2}\,.
\]
Here, we choose the constant  $c$ introduced in  $g_k=k+c$ in order to have  $\beta=0$ and consequently a good decreasing in time. This means to require $g_1+d/2=0$  and $c=-1-d/2$. Consequently, we have
\[
g_k=k-\frac{d}{2}-1\,.
\]
We introduce now
\[
a=1+d/2+r>1\,,
\]
so that $k+r=g_k+a$. 
We can now  divide by $Y_k$ in (\ref{eq:yk}), multiply by  $(1+t)^{g_k}$ and  summate on $k$  to obtain a differential inequality in $Z$ (defined in  (\ref{eq:normez})):
\begin{align*}
\frac{\d{Z}}{\d{t}} +\frac{a}{1+t}Z&\leq C\left(Z^2 +\frac{Z}{(1+t)^2} +(1+t)^{\e-1}Z^{\e+1}\right)\,.
\end{align*}
Then, we  introduce $\zeta (t)=(1+t)^a\exp\left(\frac{C}{1+t}\right)Z(t)$  and we deduce from the inequality just above 
\[
\frac{\d{\zeta}}{\d{t}}\leq \frac{C}{(1+t)^a}(\zeta^2+\zeta^{\e+1}).
\] 
Besides, $\zeta^2+\zeta^{\e+1}\leq 2\zeta(1+\zeta^\e)$ for $\e\geq 2$. Therefore
\[
\frac{1}{\zeta(1+\zeta^{\e})}\frac{\d{\zeta}}{\d{t}}\leq \frac{C}{(1+t)^a}\,,
\]
that is to say
\[
\frac{\d{}}{\d{t}}(f(\zeta(t)))\leq \frac{C}{(1+t)^a}\,,
\]
with  $f(x)=\frac{1}{\e}\ln \left( \frac{x^\e}{1+x^\e}\right)$. By integration, we  obtain
$f(\zeta(t))+\frac{C}{a-1}(1+t)^{-(a-1)}\leq f(\zeta(0))+\frac{C}{a-1}$. As $f$ is strictly increasing and one-by-one from $\rpis$ to $\reali^*_-$, if $f(\zeta(0)) +C/(a-1)$ belongs to the set on which $f^{-1}$ is well-defined, we obtain
\[
\zeta(t)\leq f^{-1}\left(f(\zeta(0)) +C/(a-1)\right)\,.
\]
But $f(\zeta(0)) +C/(a-1)\leq 0$ is only possible if $\zeta(0)$ is small enough, since $f(x)$ tends to $-\infty$ when $x$ tends to  0, and  $\zeta(0)=\exp\left({C}\right)Z(0)=\exp(C)\norma{\pi_0}_{\H{m}}$.  The smallness condition is satisfied thanks to the hypothesis \textbf{(H1)} with $0<\varepsilon_0 <f^{-1}\left(\frac{-C}{a-1}\right)$.

\paragraph{Conclusion.}
We have obtained the following inequalities
\begin{align*}
Z(t)&\leq \frac{1}{(1+t)^a}\exp\left(\frac{-C}{1+t}\right)f^{-1}\left(f(e^C Z(0)) +\frac{C}{a-1}\right)\,,\\
 Y_k(t)&\leq  (1+t)^{-g_k}Z(t)\\
&\leq \frac{1}{(1+t)^{k+r}}\exp\left(\frac{-C}{1+t}\right)f^{-1}\left(f(e^C Z(0)) +\frac{C}{a-1}\right)\,.
\end{align*}
The $\L2$ norms of the derivatives of the local solution $U$ consequently do not explode in finite time since $t\mapsto  \frac{1}{(1+t)^a}\exp\left(\frac{-C}{1+t}\right)$ do not explode in finite time. Let us assume that the regular solution exists to time $T$. Our estimates give us, for all $t\in [0,T[$, for $C_T$  not depending on $T$, 
\[
 \norma{(\pi,u-\bar u)(t)}_{\H{m}}\leq C_T\,.
\]
Since  $\spt \pi \subset B(0,R)$, our construction is possible when the norm of  $(\pi,u)$ is bounded in  $\H{m}(\reali^d)$ and
\[
  \norma{(\pi,u)(t)}_{\H{m}(B(0,R))}\leq C_T+\norma{\bar u (t)}_{\H{m}(B(0,R))} \leq K_T\,.
\]
We can associate to the constant  $K_T$ a time of existence  $T_*(K_T)$  for the local in time solution. Let  $t_1\in ]0,T[$ be such that $t\geq T-T_*(K_T)$. Introducing the solution with initial condition $(\pi(t_1), u(t_1))$,  we succeed in prolongating the solution up to time $T$, which finishes the proof.

\subsection{General Case}
\subsubsection{Local in time Existence}
As in the isentropic case, we first seek to symmetrise the system. Let us denote
\[
\pi=\sqrt{\frac{\g -1}{\g}}\left(\frac{p}{\g-1}\right)^{\frac{\g -1}{2\g}} =\sqrt{\frac{\g -1}{\g}}\left(\frac{\rho}{1- b \rho}\right)^{\frac{\g -1}{2}}\exp\left(\frac{\g-1}{2\g} \frac{s}{c_v}\right)\,.
\]
The system  (\ref{eq:euler}) can be written in variables $(\pi, u,s)$:
\begin{equation}\label{eq:noi}
\left\{
\begin{array}{rcl}
e^{s/(\g c_v)}\pt_t \pi +e^{s/(\g c_v)}u\cdot \nabla\pi +\frac{\g-1}{2}  e^{s/(\g c_v)}\pi \div u&=&-\frac{\g-1}{2}\tilde b\pi^{\frac{\g +1}{\g -1}}\div u\,, \\
\pt_t u+(u\cdot \nabla)u +\frac{\g-1}{2} e^{s/(\g c_v )}\pi \nabla \pi&=&-\frac{\g-1}{2}\tilde b\pi^{\frac{\g +1}{\g -1}}\nabla\pi\,,\\
(1+t)^{-\theta}(\pt_t s+u\cdot\nabla s)&=&0\,,
\end{array}
\right.
\end{equation}
where $\tilde b=b\left(\frac{\g-1}{4\g}\right)^{\frac{1}{\g-1}}$. We introduce furthermore a parameter  $\theta$ to be determined  so that   $(1+t)^{-\theta}s$ has a decreasing in time  similar to  the estimates obtained in the isentropic case. 

In order to obtain local existence of a solution, we construct a solution by following the same strategy as in the isentropical case and using once again a property of local uniqueness, given by Proposition \ref{prop:uniloc_nonise}.

\subsubsection{Local in time Uniqueness}
We show here a similar result to the one obtained in  Section  \ref{sec:uniloc_ise} in the isentropical case.

\begin{proposition}\label{prop:uniloc_nonise}
Let $U_0=(\pi_0,u_0, s_0)^{\mathbf{T}}\in \H{m}(\reali^d)$ and $\tilde U_0$  be two initial data for  (\ref{eq:symMUK2ise}). Let  $U=(\pi,u, s)^{\mathbf{T}}$, $\tilde U$ be the two corresponding solutions defined for $0\leq  t \leq T_0$. We assume that $\norma{\D{\tilde U}}_{\L\infty([0,T_0]\times \reali^d)}<\infty$. Let $x_0\in \reali^d$ and $R \geq 0$. We denote
\begin{align}
M&=\sup \{ e^{\frac{s}{2 \g c_v}}(\frac{\g-1}{2} \modulo{\pi}(1+ \tilde b \modulo{\pi}^{\e-1})+\modulo{u})(t,x)\,, (t,x)\in [0,T_0]\times B(x_0,R) \}\,,\label{eq:Mg}\\
C_T&=\{
(t,x)\in [0,T]\times B(x_0,R-Mt)
\} \qquad\textrm{ for }T\in[0,T_1]\,,\label{eq:ctg}
\end{align}
where $T_1=\min(T_0,\eta/M)$.

If $U_0=\tilde U_0$ on $B(x_0,R)$ then $U=\tilde U$ on $C_{T_1}$.
\end{proposition}

\begin{proof}
Let $U_0=(\pi_0,u_0, s_0), \tilde U_0=(\tilde\pi_0,\tilde u_0, \tilde s_0)$  be two initial data for (\ref{eq:noi}) such that $U_0\in \H{m}$. Let $U,\tilde U$  be the two solutions of the associated Cauchy problem. We assume that these solutions are defined on  $[0,T_0]$ with $T_0>0$.  Let furthermore $x_0\in \reali^d$, $\eta\in \rpis$ and $M$, $C_T$  be as in  (\ref{eq:Mg}) and (\ref{eq:ctg}).  Then we have 
\[
\alpha_0(U)\pt_t U +\sum_j \alpha_j(U)\pt_j U=0\,,
\]
where $\alpha_0=\mathrm{Diag}(e^{\frac{s}{\g c_v}}, 1,\ldots , 1,1)$ and, for all $j\in \{1,\ldots, d\}$, 
\[
 \alpha_j(U)=\left(
\begin{array}{cccc}
e^{\frac{s}{\g c_v}}u_j&0& e^{\frac{s}{\g c_v}}\frac{\g-1}{2}(\pi+\tilde b \pi^\e)&\\
0&u_j&&\\
e^{\frac{s}{\g c_v}} \frac{\g-1}{2}(\pi +\tilde b\pi^\e)&&&0\\
\vdots&&\ddots&0\\
0&0&&u_j
\end{array}
\right)\,.
\]

We introduce, for  $T\in [0,T_1]$,
\begin{equation}\label{eq:Inonise}
I=\int_{C_T}\left(\pt_t[(U-\tilde U)\cdot \alpha_0(U)(U-\tilde U)]+\sum_j \pt_j[(U-\tilde U)\cdot \alpha_j(U)(U-\tilde U)] \right)\d{x}\d{t}\,.
\end{equation}
Denoting $\modulo{U-\tilde U}_{0}^2=(U-\tilde U)\alpha_0(U)(U-\tilde U)$ and  $\left[ U-\tilde U\right]_t=\int_{B(x_0,\eta-Mt)}\modulo{(U-\tilde U)(t,x)}_0^2\d{x}$,  we obtain from the Stokes formula
\begin{align*}
I&= \int_{\pt C_T} \left((U-\tilde U)\cdot \alpha_0(U)(U-\tilde U)n_t +\sum_j (U-\tilde U)\cdot \alpha_j(U)(U-\tilde U) n_j\right)\d{\sigma}\\
&=\left[ U-\tilde U\right]_T-\left[ U-\tilde U\right]_0\\
& +\frac{1}{\sqrt{1+1/M^2}} \int_{\mathcal{C}}  \left(\modulo{U-\tilde U}_0 + \sum_j(U-\tilde U)\cdot \alpha_j(U)(U-\tilde U)\frac{x_j}{M\modulo{x}}\right)\d{\sigma}\,.
\end{align*}
But we have also
\begin{align*}
&\sum_j (U-\tilde U)\alpha_j(U)(U-\tilde U)\frac{x_j}{M\modulo{x}}\\
=&\modulo{U-\tilde U}_0^2 \frac{u\cdot x}{M\modulo{x}} +2e^{\frac{s}{\g c_v}}\frac{\g -1}{2}(\pi+\tilde b \pi^\e)(\pi-\tilde \pi)\frac{(u-\tilde u)\cdot x}{M\modulo{x}}\\
\leq& \modulo{U-\tilde U}_0^2 \frac{\modulo{u}}{M}+ e^{s/(2\g c_v)}\frac{\g-1}{2}(\modulo{\pi}+\tilde b\modulo{\pi}^\e)(e^{s/(\g c_v)}\modulo{\pi-\tilde \pi}^2 +\modulo{u-\tilde u}^2 )\frac{1}{M}\\
\leq& \frac{1}{M}\modulo{U-\tilde U}_0^2(\modulo{u}+e^{s/(2\g c_v)} \frac{\g-1}{2}(\modulo{\pi}+\tilde b\modulo{\pi}^\e) )\\
\leq &\modulo{U-\tilde U}_0^2\,,
\end{align*}
therefore
\[
I\geq \left[ U-\tilde U\right]_T-\left[ U-\tilde U\right]_0\,.
\]

Besides, we have 
\[
\alpha_0(U)\pt_t(U-\tilde U)+\sum_{j}\alpha_j(U)(U-\tilde U)=\sum_j \alpha_0(U)(\alpha_0(\tilde U)^{-1}\alpha_j(\tilde U)-\alpha_0(U)^{-1}\alpha_j(U))\pt_j \tilde U\,,
\]
thus
\begin{align*}
I=&\int_{C_T} (U-\tilde U)\pt_t \alpha_0(U) (U-\tilde U)+\sum_j (U-\tilde U)\pt_j \alpha_j(U)(U-\tilde U)\\
& +2\int_{C_T}\sum_j \alpha_0(U)(\alpha_0(\tilde U)^{-1}\alpha_j(\tilde U)-\alpha_0(U)^{-1}\alpha_j(U))\pt_j \tilde U\,.
\end{align*}
Let us denote 
\begin{align*}
I_1&=\int_{C_T} (U-\tilde U)\pt_t \alpha_0(U) (U-\tilde U)\,,\\
I_2&=\sum_j (U-\tilde U)\pt_j \alpha_j(U)(U-\tilde U)\,,\\
I_3&=2\int_{C_T}\sum_j \alpha_0(U)(\alpha_0(\tilde U)^{-1}\alpha_j(\tilde U)-\alpha_0(U)^{-1}\alpha_j(U))\pt_j \tilde U\,.
\end{align*}
We obtain by computing explicitly $\pt_t \alpha_0$ and $\pt_j\alpha_j$
\[
I_1=-\int_{C_T}\frac{1}{\g c_v }e^{s/(\g c_v)}(\pi-\tilde \pi)^2 (u\cdot \nabla s)\,
\]
and 
\begin{align*}
I_2&=\int_{C_T}\frac{1}{\g c_v }e^{s/(\g c_v)}(\pi-\tilde \pi)^2 (u\cdot \nabla s)+\int_{C_T}\div (u)\modulo{U-\tilde U}_0^2 \\
&\qquad +\int_{C_T}(\g-1)e^{s/(\g c_v)}(\pi-\tilde \pi)(u-\tilde u)\cdot\left[\frac{1}{\g c_v}(\pi+\tilde b \pi^\e) \nabla_s+(1+\tilde b\e \pi^{\e-1})\nabla \pi\right]\\
&\leq  -I_1+\norma{DU}_{\L\infty}(1+C(1+\norma{\pi}_{\L\infty})(1+\tilde b\e\norma{\pi^{\e-1}}_{\L\infty}))\int_0^T\left[U-\tilde U\right]_t \d{t} \,.
\end{align*}
Finally, we bound $I_3$ by computing
\begin{align*}
&\alpha_0(\tilde U)^{-1}\alpha_j(\tilde U)-\alpha_0(U)^{-1}\alpha_j(U)\\
=&-\left(
\begin{array}{c|c|c}
u_j-\tilde u_j& \frac{\g-1}{2}(\pi-\tilde \pi+\tilde b(\pi^\e-\tilde \pi^\e))\mathbf{e_j}^{\mathbf{T}}&0\\
\hline
\Psi\mathbf{e_j} &
(u_j-\tilde u_j)I_d&0\\
 \hline
 0&0&u_j-\tilde u_j
\end{array}
\right)\,,
\end{align*}
where $\Psi=\frac{\g-1}{2}(e^{\frac{s}{\g c_v}}(\pi+\tilde b\pi^\e)-e^{\frac{\widetilde s}{\g c_v}}(\tilde\pi+\tilde b\tilde \pi^\e))$.
We denote $R=\max(\norma{s}_{\L\infty},\norma{\tilde s}_{\L\infty})$. The exponential function being convex, we have:
\[
\modulo{e^{s/(\g c_v)}-e^{\tilde s/(\g c_v)}}\leq 1/(\g c_v)e^{R/(\g c_v)}\modulo{s-\tilde s}\,.
\]
So,
\begin{align*}
e^{s/(\g c_v)}(\pi+\tilde b\pi^\e)-e^{\tilde s/(\g c_v)}(\tilde\pi+\tilde b\tilde \pi^\e)&\leq  C (1+R^{\e-1})(\modulo{\pi}\modulo{s-\tilde s}+\modulo{\pi-\tilde \pi})\,,
\end{align*}
which gives us
\begin{align*}
I_3\leq&  \left(\norma{D\tilde U}_{\L\infty}(1+C(1+\norma{\pi}_{\L\infty})(1+\tilde b\e\norma{\pi^{\e-1}}_{\L\infty}))\right.\\
&\left.+\norma{DU}_{\L\infty}(1+C(1+\norma{\pi}_{\L\infty})(1+\tilde b\e\norma{\pi^{\e-1}}_{\L\infty}))\right)\int_0^T\left[U-\tilde U\right]_t \d{t}\,.
\end{align*}
Finally, we obtain
\[
\left[U-\tilde U\right]_T-\left[U-\tilde U\right]_0\leq C (1+C(1+\norma{\pi}_{\L\infty})(1+\tilde b\e\norma{\pi^{\e-1}}_{\L\infty}))\int_0^T\left[U-\tilde U\right]_t\d{t}.
\]
We conclude thanks to Gronwall Lemma
\[
\frac{1}{2}\int_{B(x_0,R-MT)}\modulo{U-\tilde U}^2(T,x)\d{x}\leq\frac{1}{2}e^{C'T} \int_{B(x_0,R)}\modulo{U_0-\tilde U_0}^2(x)\d{x}
\]
where $C'=C\norma{\D{U}}_{\L\infty(C_T)}(1+\norma{\pi}_{\L\infty(C_T)}^{\e-1})$.
\end{proof}

\subsubsection{Estimates}
The system (\ref{eq:noi})  can be written 
\begin{equation}\label{eq:nonise}
A_0\pt_tV +\sum_{j=1}^d A_j \pt_j=-B(\D{\overline V},V)-\sum C_j(\overline V)\pt_j V -\frac{\g-1}{2}\tilde b\pi^\e\left(\begin{array}{c} 
\div (w+\overline u)\\ \nabla \pi\\0 
\end{array}\right)\,,
\end{equation}
where $V=(\pi,u-\bar u,s)\in \reali^{d+2}$, $\overline V=(0,\bar u,0)$ and, denoting $(\mathbf{e_j})_{1\leq j\leq d}$  the standard orthonormal basis of  $\reali^d$,
\begin{align*}
A_0&=\mathrm{Diag}(e^{s/(\g c_v)}, 1,\ldots,1,(1+t)^{-\theta})\in \mathcal{M}_{d+2}(\reali)\,,\\
C_j&=\bar u_j A_0\,,\\ 
A_j&= \left(
\begin{array}{c|ccc|c}
 e^{s/(\g c_v)}w_j&&\frac{\g-1}{2} e^{s/(\g c_v)}\pi \mathbf{e_j}^{\mathbf{T}}&&0\\
\hline
&&&&\\
\frac{\g-1}{2} e^{s/(\g c_v)}\pi \mathbf{e_j}&&w_j I_d&&0\\
&&&&\\
\hline
0&&0&&(1+t)^{-\theta}w_j
\end{array}
\right),\\
 B&=
\left(
\begin{array}{c}
\frac{\g-1}{2} e^{s/(\g c_v)}\pi \div \bar u\\
\\
(w\cdot\nabla)\bar u\\
\\
0
\end{array}
\right).
\end{align*}
We also introduce $N_k(t)=\left(\int_{\reali^d} \D{}^k V\cdot A_0(V)\D{}^kV\d{x}\right)^{1/2}$ and $Z(t)=\sum_{k=0}^m (1+t)^{g_k}N_k(t)$ with $g_k=k+r-a$ and $r=\theta/2-d/2$, $\theta\in ]0,\min(1,\frac{\g-1}{2})]$.
In order to obtain energy estimates, we apply $\D{}^k$ to (\ref{eq:nonise}) and we multiply it by $\D{}^k V$. Then, we integrate on $\reali^d$. The additional term with respect to the Perfect Polytropic gases considered by M. Grassin \cite{Grass} is now
\[
F^*(\D{\overline V}, \D{V}, V)=\frac{\g-1}{2}\tilde b\pi^\e \left( 
\div (w+\overline u), \nabla \pi , 0 
\right)\in \reali^{d+2}\,.
\]
With the notations $U=(\pi,w)\in \reali^{d+1}$ and $F$ as in \Ref{eq:F}, the last component of $F^*$ being 0, we have 
\[
\int_{\reali^d} \pt^k V \pt^k(F^*(\D{\overline V},\D{V}, V) )\d{x}= \int_{\reali^d} \pt^k U\pt^k(F(\D{\overline U},\D{U}, U))\d{x}\,.
\]
Using the estimates (\ref{eq:Iise})--(\ref{eq:J1})--(\ref{eq:J2})  we finally get an estimate  on $Y_k$.  The definition of the norm is slightly changed  with respect to isentropical case, however, for all  $v=(v_1\ldots,v_{d+1})\in \reali^{d+1}$, if $v^*=(v_1,\ldots, v_{d+1}, z)\in \reali^{d+2}$, we have $\norma{v}_{2}\leq e^{\norma{s_0}_{\L\infty}/(\g c_v)} {}^t\!v^* A_0 v^*$. Consequently $Y_k\leq e^{\norma{s_0}_{\L\infty}/(\g c_v)} N_k$ and the estimates on $Y_k$ obtained in the isentropical case give  an estimate on  $N_k$  in the general case.

Finally, we obtain, adding the estimate on $F$ obtained in the isentropical case to the estimates from M. Grassin in the general case:
\begin{align*}
\frac{1}{2}\frac{\d{N_k^2}}{\d{t}}+\frac{k+r}{1+t}N_k^2 & 
\leq C(1+t)^\beta N_k^2 Z +C' (1+t)^{-g_k-2}N_k Z +C (1+t)^{\beta+(\beta+1)(\e-1)}N_k^2 Z^\e\\
& +C(1+t)^{(\beta+1)(\e-1)-1}N_k^2 Z^{\e-1} +C(1+t)^{(\beta+1)(\e-1)-g_k-2}N_k  Z^\e \\
&+C\sum_{\xi\in {E}_k} N_k {Z}^{2+\xi}(1+t)^{-g_k+\beta+\xi(\beta+1+\theta/2)} 
\end{align*}
where $\beta=-g_1-\frac{d}{2}$ and
\[
E_k=\lb 0,k\rb\cup \left\{\frac{l}{k-1}\,;\;l\in \lb 1, k-1 \rb \right\}\,.
\]
Then, we choose $a$ so that  $\beta=0$, i.e.
\[
a=1+\theta/2>1\,\textrm{ with } \theta\in  ]0,\min(1,\frac{\g-1}{2})]\,.
\]
Next, we simplify by  $N_k$,  we multiply by $(1+t)^{g_k}$ and we summate on $k$ to obtain
\begin{align*}
\frac{\d{Z}}{\d{t}} +\frac{a}{1+t}Z&\leq C\left(Z^2 +\frac{Z}{(1+t)^2} +(1+t)^{\e-1}Z^{\e+1} + Z^{2+m}(1+t)^{am} \right)\,.
\end{align*}
We denote now $\zeta (t)=(1+t)^a\exp\left(\frac{C}{1+t}\right)Z(t)$  and we deduce from the inequality just above
\[
\frac{\d{\zeta}}{\d{t}}\leq \frac{C}{(1+t)^a}(\zeta^2+\zeta^{\e+1}+\zeta^{m+2}).
\]
We conclude in the same way we did in the isentropical case, replacing $\e $ by  $\e^*=\max(\e,m+1)\geq 2$, since  $\zeta^2+\zeta^{\e+1}+\zeta^{m+2}\leq 2(\zeta^2+\zeta^{\e^*+1})$.

\section{Technical tools}\label{sec:tec}
\subsection{Lemmas \ref{lem:I} and \ref{lem:J2}}\label{sec:lem}
We show here Lemma \ref{lem:I},  which states, with the notations introduced in Section  \ref{sec:isen}:
\[
\modulo{I}\leq C \norma{\pi}_{\L\infty}^{\e-1}\norma{\D{U}}_{\L\infty}\norma{\D{}^k U}_{\L2}^2\,.
\]

\begin{proofof}{Lemma \ref{lem:I}}
\begin{description} 
\item[If $k=0$,] $I=-\check b \int_{\reali^d} \pi^\e(\pi \div w+w\cdot \nabla w)\d{x}$. By integration by parts, we obtain
\begin{align*}
I&=-\check b \frac{\e}{\e+1}\int_{\reali^d}\pi^{\e+1}\div w\d{x}\leq C \norma{\pi^{\e-1}}_{\L\infty}\norma{\pi}_{\L2}^2 \norma{\div w}_{\L\infty}\,.
\end{align*}
\item[If $k\geq 1$,] $I$  is such that
\begin{align*}
I&= -\check b \sum_{\alpha\in \naturali^d\mid \modulo{\alpha}_1=k} \int_{\reali^d} \pt^\alpha\pi \pt^\alpha(\pi^\e \div w)+\sum_j\pt^\alpha w_j \pt^\alpha(\pi^\e \pt_j \pi)\,.
\end{align*}
Expanding, we find
\begin{align*}
I=-\check b \sum_{\alpha\in \naturali^d\mid \modulo{\alpha}_1=k}\int_{\reali^d} \pt^\alpha(\pi) \pi^\e\pt^\alpha( \div w)+\sum_j\pt^\alpha( w_j) \pi^\e\pt^\alpha( \pt_j \pi) +\Sigma\,,
\end{align*}
where $\Sigma$  is a sum of terms as $\int_{\reali^d} \pt^{\alpha_0} U \pt^{\alpha_1} (\pi^\e)\pt^{\alpha_2} U$, where $U$ is any of its component and  $\modulo{\alpha_0}_1=k$, $\modulo{\alpha_1}_1=l$, $\modulo{\alpha_2}_1=k+1-l$, with $l\in\lb 1,k\rb$  so that the derivatives are  of order less than $k$.  We treat first one of the terms of the preceding sum for a $d$-uplet $\alpha\in\naturali^d$, of size $k$, that is to say $\modulo{\alpha}_1=\alpha_1+\ldots+\alpha_d=k$. By integration by parts, we find
\begin{align*}
&\int_{\reali^d} \pt^\alpha(\pi) \pi^\e\pt^\alpha( \div w)+\sum_j\pt^\alpha( w_j) \pi^\e\pt^\alpha( \pt_j \pi)\\
=&\int_{\reali^d} \pi^\e\sum_j \pt_j \left(\pt^\alpha(\pi) \pt^\alpha(  w_j)\right)\\
=&-\sum_j \int_{\reali^d} \pt_j(\pi^\e)\pt^\alpha(\pi) \pt^\alpha(  w_j)\,,
\end{align*}
Hence, $I$ is a sum of terms as $\int \pt^k U \pt^l(\pi^\e) \pt^{k-l+1}U$ where $1\leq l\leq k-1$. 

Note that we used here the notation after which  $\pt^k U$ means $\pt^{\alpha} U$ for a given  $\alpha\in \naturali^d$ such that $\modulo{\alpha}_1=\sum \alpha_j =k$.  Furthermore, we write $U$ for any of its component. Consequently, $\pt^k U$ can mean $\pt^\alpha w_j$.  We will use by now this notation.

If $k\geq 1$ and  $l\neq1$, $l\neq k$,  using the Hölder inequality and the Gagliardo-Nirenberg Lemma \ref{lem:GN}, we have:
\begin{align*}
\!\!\!\!\int_{\reali^d} \pt^k U\pt^l(\pi^\e)\pt^{k-l+1} U\d{x}&\leq \norma{\pt^k U}_{\L2}\norma{\pt^l(\pi^\e) \pt^{k-l+1}U}_{\L2}\\
&\leq \norma{\pt^k U}_{\L2}\norma{\pt^l(\pi^\e)}_{\L{{2\frac{k-1}{l-1}}}} \norma{\pt^{k-l+1}U}_{\L{{2\frac{k-1}{k-l}}}}\\
&\leq \norma{\pt^k U}_{\L2}\norma{\pt(\pi^\e)}_{\L\infty}^{{1-\frac{l-1}{k-1}}} \norma{\pt^k (\pi^\e)}_{\L2}^{\frac{l-1}{k-1}}\\
&\qquad \times \norma{\pt U}_{\L\infty}^{1-\frac{k-l}{k-1}} \norma{\pt^k U}_{\L2}^{\frac{k-l}{k-1}}\,.
\end{align*}
We use next Lemma \ref{lem:fp} and the inequality
\begin{align*}
\norma{\pt (\pi^\e)}_{\L\infty}&=\norma{\e \pi^{\e-1}\pt \pi}_{\L\infty}\leq C\norma{\pi}_{\L\infty}^{\e-1} \norma{\pt \pi}_{\L\infty}\,,
\end{align*}
to obtain
\begin{align*}
\int_{\reali^d} \pt^k U\pt^l(\pi^\e)\pt^{k-l+1} U\d{x}&\leq C \norma{\pt^k U}_{\L2}^2 \norma{\pt U}_{\L\infty} \norma{\pi}_{\L\infty}^{\e-1}\,.
\end{align*}

If $k\geq 1$ and $l=k$,  we have to estimate 
\begin{align*}
\int_{\reali^d}\pt^k U\pt^k(\pi^\e)\pt U&\leq \norma{\pt^k U}_{\L2}\norma{\pt^k(\pi^\e)}_{\L2}\norma{\pt{ U}}_{\L\infty}\\
&\leq \norma{\pt^k U}_{\L2}\norma{\pi}^{\e-1}_{\L\infty}\norma{\pt^k \pi}_{\L2}\norma{\pt{U}}_{\L\infty}\,. 
\end{align*}

If $k\geq 1$ and $l=1$, we have to estimate
\begin{align*}
\int_{\reali^d}\pt^k U\pt (\pi^\e)\pt^k U&\leq \norma{\pt^k U}_{\L2}\norma{\pt^k(\pi^\e)}_{\L2}\norma{\pt{ U}}_{\L\infty}\\
&\leq \norma{\pt^k U}_{\L2}^2 \norma{\pt^1 (\pi^{\e})}_{\L\infty}\,.
\end{align*}
The inequality $\norma{\pt^1(\pi^\e)}_{\L\infty}\leq C \norma{\pi}_{\L\infty}^{\e-1} \norma{\pt^1 \pi}_{\L\infty}$ allows us to conclude.
\end{description}
\end{proofof}

We prove now Lemma \ref{lem:J2}, which states, with the notations introduced in  Section  \ref{sec:isen}
\[
J_2\leq C(1+t)^{d_k} \norma{\D{}^k U}_{\L2} Z^\e\,,
\]
where $d_k=(\beta+1)(\e-1)-g_k-2$ and $\beta=-g_1-\frac{d}{2}$.

\begin{proofof}{Lemma \ref{lem:J2}}
For $k=0$, $J_2=0$. We are considering here $k\geq 1$ ; $J_2$  is then a sum of terms  $\int \pt^k U\pt^l(\pi^\e)\pt^{k-l+1}\overline u\d{x}$ for $0\leq l\leq k-1$. The choice of initial conditions gives us $U\in \H{m}$, but we do not know if  $\D{}^2U\in \L\infty$.  So we can distinguish two cases: $m>2+d/2$ and $\D{}^2 U\in \L\infty$, or $m\leq 2+d/2$.
\vspace{0.2cm}

\textbf{Case $m>2+d/2$. } 
We study now different cases, after the values of  $k$ and $l$.
\begin{enumerate}
\item If  $k\geq 1$ and $l=0$, we use  Proposition \ref{prop:ubar} and Lemma \ref{lem:fp} to obtain the estimate
\begin{align*}
\int_{\reali^d} \pi^\e\pt^k \pi\pt^{k+1} \overline u\d{x} 
&\leq \norma{\pi^\e}_{\L\infty} \norma{\pt^k U}_{\L2} \norma{\pt^{k+1} \overline u}_{\L2} \\
&\leq C (1+t)^{(\beta+1)\e} Z^\e \norma{\D{}^k U}_{\L2} (1+t)^{d/2-k-2}\\
&\leq C(1+t)^{d_k}\norma{\D{}^k U}_{\L2}Z^\e\,,
\end{align*}
since $k-d/2+2=\beta+1+g_k+2$ we are done.
\item If $k\geq 2$ and $l=1$, we have to estimate
\begin{align*}
\int_{\reali^d} \pi^{\e-1}\pt \pi \pt^k\pi \pt^k \overline u\d{x}
&\leq \norma{\pt^k \pi}_{\L2}\norma{\pt^k \overline u}_{\L2}\norma{\pi}_{\L\infty}^{\e-1}\norma{\pt \pi}_{\L\infty}\\
&\leq  C(1+t)^{d/2-k-1} (1+t)^{(\beta+1)(\e-1)+\beta}Z^\e \norma{\pt^k U}_{\L2}\,,
\end{align*}
which concludes that case, since $d/2-k-1+\beta+(\beta+1)(\e-1)=(\beta+1)(\e-1)-g_k-2$.
\item If $k=3$ and $l=2$, 
\begin{align*}
\int_{\reali^d}\pt^3 \pi \pt^2(\pi^\e)\pt^2 \overline u \d{x} 
&\leq \norma{\pt^k \pi}_{\L2}\norma{\pt^2\overline u}_{\L\infty}\norma{\pt^2(\pi^\e)}_{\L2}\\
&\leq (1+t)^{-3+(\beta+1)(\e-1)-g_2}\norma{\pt^k U}_{\L2} Z^\e\,,
\end{align*}
and we are done as $g_2+3=g_3 +2$.
\item If $k>3$ and $l\in \lb 2,k-1\rb$, we use Lemma \ref{lem:gn2}. Denoting $q=2\frac{k-3}{l-2}$ and $q'=2\frac{k-3}{k-l-1}$ so that $1/q+1/q'=1/2$, we obtain:
\begin{align*}
&\int_{\reali^d} \pt^k \pi \pt^l(\pi^\e) \pt^{k-l+1} \overline u \d{x}\\
&\leq \norma{\pt^k U}_{\L2} \norma{\pt^l (\pi^\e)}_{\L{q}} \norma{\pt^{k-l+1} \overline u}_{ \L{{q'}}}\\
&\leq  \norma{\pt^k U}_{\L2} \norma{\D{}^2 (\pi^\e)}_{\L\infty}^{1-2/q} \norma{ \D{}^{k-1} (\pi^\e)}_{\L2}^{2/q} \norma{\D{}^2 \overline u}_{\L\infty}^{1-2/q'} \norma{\D{}^{k-1}\overline u}_{\L{2}}^{2/q'}\,.
\end{align*}
Since
\begin{align*}
\norma{\D{}^2\overline u}_{\L\infty}&\leq C (1+t)^{-3}\,,& \norma{\D{}^{k-1} \overline u}_{\L2}&\leq C(1+t)^{d/2-k}\,,\\
\end{align*}
and
\begin{align*}
\norma{\D{}^2 (\pi^\e)}_{\L\infty}&\leq C\left(\norma{\pi}_{\L\infty}^{\e-2}\norma{\D{\pi}}_{\L\infty}^2+ \norma{\pi}_{\L\infty}^{\e-1} \norma{ \D{}^2 \pi}_{\L\infty}\right)\\
&\leq C (1+t)^{(\beta+1)\e-2}Z^\e\,,\\
\norma{\D{}^{k-1}(\pi^\e)}_{\L2}
&\leq C\norma{\pi}_{\L\infty}^{\e-1}\norma{\D{}^{k-1}\pi}_{\L2}\\
&\leq C(1+t)^{(\beta+1)(\e-1)-g_k +1}Z^\e\,.
\end{align*}
we obtain
\begin{align*}
J_2 & \leq C (1+t)^{m_k}\norma{\D{}^k \pi}_{\L2}Z^\e\,,
\end{align*}
where
\begin{align*}
m_k&=-3\left(1-\frac{2}{q}\right)-\frac{2}{q}(1+\beta+g_k) +\frac{2}{q}((\beta+1)\e-2) \\
&+\left(1-\frac{2}{q}\right) ((\beta+1)(\e-1)-g_k+1)\\
&= (\beta+1)(\e-1)-g_k-2\\
&=d_k\,.
\end{align*}
\end{enumerate}
\vspace{0.2cm}

\textbf{Case $m\leq 2+d/2$. }\\
First, we note that the computations of the case $k\geq 1$ and $l=0$, $k\geq 2$ and $l=1$,  $k=3$ and $l=2$ are similar. There remains to treat the case $k>3$ and $2\leq l\leq k-1$. Since $k\leq m\leq 2+d/2$ we have necessarily $k\leq 3$, if $d=1$, $2$ or $3$, and we are done. \\
We assume now $d\geq 4$. 
Let us denote $h=\frac{1}{2}(k+1+d/2)>2$. Then we have $h\leq m$ and 
\begin{align*}
0&< h-l\leq \frac{d-1}{2}\,,& \frac{1}{2}&\leq h-(k+1-l)\leq \frac{d-1}{2}\,.
\end{align*}
We introduce $h_1=h-l$, $h_2=h-(k-l+1)$ and $1/q_1=1/2-h_1/d$,  $1/q_2=1/2-h_2/d$. Therefore $1/q_1+1/q_2=2$, which allows to use Hölder inequality
\begin{align*}
J_2(k,l)=\int_{\reali^d} \pt^k \pi \pt^l (\pi^\e)\pt^{k-l+1} \overline u\d{x}&\leq \norma{\pt^k \pi}_{\L2}\norma{\pt^l( \pi^\e)}_{\L{{q_1}}} \norma{\pt^{k-l+1} \overline u}_{\L{{q_2}}}\,. 
\end{align*}
Next, we apply Lemma \ref{lem:dn} to find
\[
J_2(k,l) \leq C\norma{\pt^k \pi}_{\L2}\norma{\D{}^n (\pi^\e)}_{\L2} \norma{\D{}^h \overline u}_{\L2}\,.
\]
Finally, we use Lemma \ref{lem:fp} and Proposition \ref{prop:ubar}
\[
J_2(k,l)\leq C\norma{\pt^k U}_{\L2} (1+t)^{d/2-h-1-g_h +(\beta+1)(\e-1)}Z^\e\,.
\]
As $d/2-h-1-g_h+(\beta+1)(\e-1)=(\beta+1)(\e-1)-g_k-2$, we are done.
\end{proofof}

\subsection{The Gagliardo-Nirenberg inequality and its consequences}
\subsubsection{The Gagliardo-Nirenberg inequality}
\begin{lemma}[Gagliardo-Nirenberg]\label{lem:GN} (See \cite[Prop. 3.5, p. 4]{Taylor3})
Let $r>0$, $i \in[0,r]$ and $z\in (\L\infty\cap\H{r})(\reali^d)$. Then $\pt^i z\in \L{{2r/i}}(\reali^d)$ and
\[
\norma{\pt^i z}_{\L{{{2r}/{i}}}}\leq C_{i,r} \norma{z}_{\L\infty}^{1-{i/r}} \norma{\D{}^r z}_{\L2}^{i/r}\,.
\]
\end{lemma}

We deduce easily from Lemma \ref{lem:GN} the following result.
\begin{lemma}\label{lem:gn2}
Let $z\in \H{m}$ be such that $\D{}^2 z\in \L\infty$, then for all $k\in [4,m]$, for all $i\in [2,k]$, we have $\D{}^i z\in \L{q}$ for $q=2\frac{k-3}{i-2}$ and
\[
\norma{\pt^i z}_{\L{q}}\leq C\norma{\D{}^2 z}_{\L\infty}^{1-2/q}\norma{\D{}^{k-1} z}_{\L2}^{2/q}\,.
\]
\end{lemma}

Thanks to the Sobolev imbedding (see \cite[p.4]{Taylor3}) and Lemma \ref{lem:GN} we also prove 
\begin{lemma}\label{lem:dn}
Let $\ell\in ]0,d/2[$ and $1/q=1/2-\ell/d$. There exists $C>0$ depending on $\ell,q,d$ such that for all $z\in \H{\ell}(\reali^d)$ we have
\[
\norma{z}_{\L{q}}\leq C \norma{\D{}^\ell z}_{\L2}\,.
\]
\end{lemma}

\begin{proof}
The space $\H{\ell}(\reali^d)$ is endowed with the norm  $\norma{\cdot}_{\L2}+\norma{\D{}^\ell\cdot}_{\L2}$. The Sobolev imbedding between  $\H{\ell}$ and $\L{q}$  can the be written, for a given  $C>0$,
\begin{align}
\norma{z}_{\L{q}}\leq C(\norma{z}_{\L2}+\norma{\D{}^\ell z}_{\L2})\,, \qquad \textrm{  for all } z\in \H{\ell}\,.\label{inj}
\end{align}
Let us define now, for  $z\in \H{\ell}$, $\lambda\in \rpic$, the function $z_\lambda\in \H\ell$ such that $z_\lambda(x)=z(\lambda x)$. Applying  \Ref{inj} to $z_\lambda$,  and noting that  
\begin{align}
\norma{z_\lambda}_{\L{q}}&=\lambda^{\frac{-d}{q}}\norma{z}_{\L{q}}\,,&\norma{\D{}^\ell z_\lambda}_{\L2}&=\lambda^{\ell - \frac{d}{2}} \norma{\D{}^\ell z}_{\L2}\,,\label{zlambda}
\end{align}
we obtain
\[
\norma{z}_{\L{q}}\leq C\lambda^{\frac{d}{q} -\frac{d}{2}} (\norma{z}_{\L2}+\lambda^\ell\norma{\D{}^\ell z}_{\L2})\,,
\]
where, by definition $\frac{d}{q} -\frac{d}{2}={-\ell}$. Consequently, introducing $\lambda=\left(\frac{\norma{z}_{\L2}}{\norma{\D{}^\ell z}_{\L2}} \right)^{1/\ell}$, we have
\[
\norma{z}_{\L{q}}\leq 2C\norma{\D{}^\ell z}_{\L2}\,.
\]
\end{proof}

Similarly,
\begin{lemma}\label{lm:inj} 
Let $p>d/2$. There exists $C>0$ such that for all $z\in \H{p}(\reali^d)$
\[
\norma{z}_{\L\infty}\leq C\norma{z}_{\L2}^{1-\frac{d}{2p}}\norma{ \D{}^p z}_{\L2}^{\frac{d}{2p}}\,.
\]
\end{lemma}
\begin{proof}
We use now the continuous imbedding $\H{p}(\reali^d)\subset \L\infty(\reali^d)$.  Thus, there exists $C>0$ such that 
\begin{align*}
\norma{z}_{\L{\infty}}\leq C(\norma{z}_{\L2}+\norma{\D{}^p z}_{\L2})\,, \qquad \textrm{  for all } z\in \H{p}\,.
\end{align*}
Applying the inequality to   $z_\lambda : x \mapsto z( \lambda x)$, we obtain, since  $z_\lambda$ satisfies $\norma{z_\lambda}_{\L\infty}=\norma{z}_{\L\infty}$ and \Ref{zlambda},
\[
\norma{z}_{\L\infty}\leq C \lambda^{\frac{-d}{2}} (\norma{z}_{\L2}+\lambda^p\norma{\D{}^p z}_{\L2})\,.
\]
Taking $\lambda=\left(\frac{\norma{z}_{\L2}}{\norma{\D{}^p z}_{\L2}} \right)^{1/p}$, we have finished the proof.
\end{proof}

\subsubsection{Estimates}
\begin{lemma}\label{lem:Z}
Let $m>1+d/2$, $U\in \H{m}(\reali^d)$, $r, a\in \reali$  and  $Z$  be the norm defined by (\ref{eq:normez}) :
\[
Z(t)=\sum_{k=0}^m (1+t)^{g_k}\norma{\D{}^k U(t)}_{\L2}\,,
\]
 with $g_k=k+r-a$. Then we have:
\begin{enumerate}
\item  $\norma{U(t)}_{\L\infty}\leq C(1+t)^{\beta+1}Z(t)$,
\item $\norma{\D{U}(t)}_{\L\infty}\leq C(1+t)^{\beta}Z(t)$,
\item If $m>2+d/2$, then $\norma{\D{}^2 U(t)}_{\L\infty}\leq C(1+t)^{\beta-1}Z(t)$,
\end{enumerate}
with $\beta =-g_1-d/2$.
\end{lemma}
\begin{proof}
1. Applying Lemma \ref{lm:inj} to  $U$, we obtain:
\[
\norma{U}_{\L\infty}\leq C\norma{U}_{\L2}^{1-\frac{d}{2m}}\norma{ \D{}^m U}_{\L2}^{\frac{d}{2m}}\,.
\]
But we also have
\begin{align*}
\norma{U}_{\L2}&\leq (1+t)^{-g_0}Z\,,& \norma{\D{}^m U}_{\L2}&\leq (1+t)^{-g_m}Z\,,
\end{align*}
so 
\[
\norma{U}_{\L\infty} \leq C(1+t)^sZ\,,
\]
where $s=-g_0(1-\frac{d}{2m})-(g_0+m)\frac{d}{2m}=-g_0 -\frac{d}{2}=\beta+1$.

2. Applying Lemma \ref{lm:inj} to $\D{U}$, with $p=m-1$, we have
\[
\norma{\D{U}}_{\L\infty}\leq C\norma{\D{U}}_{\L2}^{1-\frac{d}{2(m-1)}}\norma{ \D{}^m U}_{\L2}^{\frac{d}{2(m-1)}}\,.
\]
In the same way as before, we obtain
\[
\norma{\D{U}}_{\L\infty}\leq C(1+t)^sZ
\]
where 
\begin{align*}
s&=-g_1(1-\frac{d}{2(m-1)})-(g_1+m-1)\frac{d}{2(m-1)}\\
&=-g_1-\frac{d}{2}=\beta\,.
\end{align*}

3. Applying Lemma  \ref{lm:inj} to  $\D{}^2U$ with $p=m-2$, which is possible since $m-2>d/2$, we finally prove the third inequality.
\end{proof}

\begin{lemma}\label{lem:pdt}
Let $f,\phi \in \H{m}\cap\L\infty(\reali^d)$, let $\alpha\in \naturali^d$ such that $\modulo{\alpha}=k\leq m$. Then 
\[
\norma{\pt^\alpha(f\phi)}_{\L2}\leq C(\norma{f}_{\L\infty}\norma{\D{}^k \phi}_{\L2}+  \norma{\phi}_{\L\infty}\norma{\D{}^k f}_{\L2} )\,.
\]
\end{lemma}

\begin{lemma}\label{lem:fp}
Let $f\in (\L\infty \cap\H{m})(\reali^d)$. If $\e\in \naturali$ and $\e\geq 2$, or $\e\in \reali$ and $\e\in [m,+\infty[$, we have $f^\e\in \H{m}(\reali^d)$ and, for all $\alpha\in \naturali^d$ such that $\modulo{\alpha}=k\leq m$, we have
\begin{equation}\label{eq:lemfp}
\norma{\pt^\alpha f^\e}_{\L2}\leq C \norma{f}_{\L\infty}^{\e-1} \norma{\D{}^k f}_{\L2}\,,
\end{equation}
where $C>0$  is a constant independent from  $f$, $\alpha$, $\e$.
\end{lemma}

\begin{proof}
For $\e\in \naturali$, $\e\geq 2$,  we proceed by iteration on $\e$, using Lemma \ref{lem:pdt}.

For $\e\in [m,+\infty[$, we have
\[
\pt^{\alpha}(f^\e(x))=\sum_{1\leq j\leq k} \sum _{ \beta_1+\ldots+\beta_j =\alpha \atop  \modulo{\beta_i}=b_i\geq 1} c_{\alpha,\beta}f^{\e-j} \pt^{\beta_1}f \ldots \pt^{\beta_j} f\,.
\]
Then we take $\pt^{\beta_i} f\in \L{{2\frac{k}{b_i}}}$, we apply Hölder inequality and Lemma \ref{lem:GN} to obtain:
\begin{align*}
\norma{\pt^{\alpha}(f^\e)}_{\L2}&\leq C \sum_{j=1}^k \norma{f^{\e-j}}_{\L\infty}\prod_{i=1\atop \sum b_i=k}^j \norma{\pt^{b_i}_{\beta_i} f}_{\L{{2\frac{k}{b_i}}}}\\
&\leq C\sum_{j=1}^k \norma{f^{\e-j}}_{\L\infty}\prod_{i=1 \atop \sum b_i=k}^j \norma{f}_{\L\infty}^{1-b_i/k}\norma{\D{}^k f}_{\L2}^{b_i/k}\\
&\leq C\norma{f}^{\e-1}_{\L\infty} \norma{\D{}^k f}_{\L2}\,,
\end{align*}
using besides that for all $j\in \lb 1,m\rb$, $\norma{f^{\e-j}}_{\L\infty}\leq \norma{f}^{\e-j}_{\L\infty}$ since  $\e-j\geq 0$.
\end{proof}

{\small \paragraph*{Acknowledgement :} The author thanks Sylvie Benzoni-Gavage for initiating this work and for providing useful advises.

\bibliographystyle{abbrv}
\bibliography{ref}}

\end{document}